\documentclass{amsart}

\usepackage{amssymb,graphicx,color}

\usepackage[all]{xy}
\SelectTips{cm}{}

\usepackage[dvipsnames]{xcolor}

\usepackage{hyperref}
\hypersetup{colorlinks=true,citecolor=brown,linkcolor=brown,urlcolor=black,filecolor=black}

\numberwithin{equation}{section}
\allowdisplaybreaks

\theoremstyle{plain}
\newtheorem{theorem}{Theorem}[section]
\newtheorem{lemma}[theorem]{Lemma}
\newtheorem{corollary}[theorem]{Corollary}
\newtheorem{proposition}[theorem]{Proposition}

\theoremstyle{definition}
\newtheorem{definition}[theorem]{Definition}
\newtheorem{example}[theorem]{Example}

\theoremstyle{remark}
\newtheorem{remark}[theorem]{Remark}

\newtheorem*{acknowledgments}{Acknowledgments}

\newcommand\nc{\newcommand}
\nc\rnc{\renewcommand}

\nc{\GMb}{\color[rgb]{.7,0,.7}}
\nc{\KHb}{\color[rgb]{0,0,.7}}
\nc{\KHe}{\normalcolor{}}
\nc{\GMe}{\normalcolor{}}
\nc\GMnote[1]{\marginpar{\GMb \tiny #1 \GMe}}
\nc\KHnote[1]{\marginpar{\KHb \tiny #1 \KHe}}
\nc{\GMcut}[1]{\marginpar{\GMb \textbf{\footnotesize CUT}: \Tiny #1  \GMe}}
\nc{\KHcut}[1]{\marginpar{\KHb \textbf{\footnotesize CUT}: \Tiny #1  \KHe}}
\nc\KH[1]{{\KHb #1}}
\nc\GM[1]{{\GMb #1}}
\nc\blue[1]{{\color{blue}#1}}
\nc\darkblue[1]{{\textcolor[rgb]{0,0,.5}{#1}}}

\nc\id{\operatorname{id}}
\nc\ad{{\operatorname{ad}}}
\nc\Ad{{\operatorname{Ad}}}
\nc\abel{{\operatorname{ab}}}
\nc\End{\operatorname{End}}
\nc\Aut{\operatorname{Aut}}
\nc\Out{\operatorname{Out}}
\nc\IAut{\operatorname{IAut}}
\nc\opn{\operatorname}
\nc\Hom{\operatorname{Hom}}
\nc\inn{\operatorname{inn}}
\nc\ab{{\mathrm{ab}}}
\nc\Ker{\operatorname{Ker}}
\nc\Br{\mathrm{Br}}
\nc\Lie{\mathrm{Lie}}
\nc\Der{\mathrm{Der}}
\nc\rank{\opn{rank}}
\nc\Mat{\opn{Mat}}
\nc\Wh{\opn{Wh}}
\nc\proj{\mathrm{proj}}
\nc\gr{\mathrm{gr}}
\nc\grb{\gr_\bu}
\nc\plim{\varprojlim}
\nc\ilim{\varinjlim}

\nc\modN {{\mathbb N}}
\nc\modQ {{\mathbb Q}}
\nc\R{{\mathbb R}}
\nc\modZ {{\mathbb Z}}
\nc\Q{\mathbb{Q}}
\nc\Vect{\mathbf{Vect}}
\nc\Grp{\mathbf{Grp}}
\nc\NS{\mathbf{NS}}
\nc\eNs{\mathbf{eNs}}
\nc\egL{\mathbf{egL}}

\nc\modR {{\mathcal R}}
\nc\modL {{\mathcal L}}
\nc\modC {{\mathcal C}}
\nc\modE {{\mathcal E}}
\nc\modV {{\mathcal V}}
\nc\modT {\mathcal{T}}
\nc\modH {\mathcal{H}}
\nc\modM {\mathcal{M}}
\nc\modI {\mathcal{I}}
\nc\modA {\mathcal{A}}
\nc\modB {\mathcal{B}}
\nc\modD {\mathcal{D}}
\nc\bC{\bar{C}}
\nc\modG {\mathcal{G}}
\nc\tH{\tilde{\modH }}
\nc\modF {\mathcal{F}}
\nc\modP {\mathcal{P}}
\nc\modS {\mathcal{S}}
\nc\Z{\modZ }
\nc\ZF{\modZ [F]}
\nc\ZH{\modZ [H]}

\nc\g{{\mathfrak g}}

\nc\xto[1]{{\overset{#1}{\longrightarrow}}}
\nc\yto[1]{{\underset{#1}{\longrightarrow}}}
\nc\xyto[2]{{\overset{#1}{\underset{#2}{\longrightarrow}}}}
\nc\simeqto{\overset{\simeq}{\rightarrow }}
\nc\trl{\triangleleft}
\nc\trll{\,\triangleleft\,}
\nc\trr{\triangleright}
\nc\trrr{\,\triangleright\,}
\nc\ct{\overset{\cong}{\to}}
\nc\onto\twoheadrightarrow
\nc\projto{\underset{\text{proj}}{\longrightarrow}}
\nc\ho{\mathbin{\hat\otimes}}
\nc\la{\langle}
\nc\ra{\rangle}
\nc\lala{\langle\!\langle}
\nc\rara{\rangle\!\rangle}
\nc\mt{\mapsto}
\nc\rt{\rtimes}
\nc\bu{\bullet} \nc\Lbu{L_\bu}

\nc\fig[1]{Figure~\ref{#1}}
\nc\FI[2]{\begin{figure}  \begin{center}\input{#1.pstex_t}\end{center} \caption{#2}  \label{#1}  \end{figure}}
\nc\FIGURE[2]{\begin{figure}  \boxed{\tt fig: #1}  \caption{#2}  \label{fig:#1}  \end{figure}}

\nc\ul{\underline}
\nc\bA{\bar{A}}		\nc\bAk{\bA_k}
\nc\bK{\bar K}		\nc\bKb{\bK_\bu}	\nc\bKbu{\bK_\bu}
\nc\bG{\bar G}		\nc\bGb{\bG_\bu}
\nc\ti{\tilde}
\nc\wh{\widehat}
\nc\wt{\widetilde}
\nc\ol{\overline}
\nc\ch{\check}

\nc\zzzvert {\;|\;}
\nc\np{\newpage}
\rnc\np{}
\nc\xym{\xymatrix}
\nc\lto{\longrightarrow}

\begin{document}

\title[Generalized Johnson homomorphisms]{{Generalized Johnson homomorphisms\\for extended N-series}}

\author{Kazuo Habiro}
\address{Research Institute for Mathematical Sciences, Kyoto University, Kyoto 606-8502, Japan}
\email{habiro@kurims.kyoto-u.ac.jp}

\author{Gw\'ena\"el Massuyeau}
\address{IRMA,  Universit\'e de Strasbourg \& CNRS,   67084 Strasbourg, France
\newline  \& \ {IMB}, Universit\'e  Bourgogne Franche-Comt\'e \& CNRS, 21000 Dijon, France}
\email{{gwenael.massuyeau@u-bourgogne.fr}}

\date{September 23, 2017 (First version: July 24, 2017)}


\maketitle

\begin{abstract}
  The Johnson filtration of the mapping class group of a compact,
  oriented surface is the
 {descending}   series consisting of the kernels of
  the actions on the nilpotent quotients of the fundamental group of
  the surface.  Each term of the Johnson filtration admits a Johnson
  homomorphism, whose kernel is the next term in the filtration.
In this paper, we consider a general situation where a group acts on a
group with a filtration called an \emph{extended N-series}.  We
develop a theory of Johnson homomorphisms in this general setting,
including many known variants of the original Johnson homomorphisms
{as well as several new variants.}
\end{abstract}

\section{Introduction}    \label{sec:introduction}

In the late seventies and eighties, Johnson studied the algebraic structure
of the mapping class group of a compact, oriented surface $\Sigma$ by examining its
{action} on the lower central series of $\pi_1(\Sigma)$ \cite{Johnson_survey}.
He introduced a filtration of the mapping class group, which is
now called the {\em Johnson filtration}, and defined homomorphisms on
the terms of this filtration, called the {\em Johnson homomorphisms}.
 His study was preceded by Andreadakis' work on the automorphism
group of a free group \cite{Andreadakis}, and {further}
developed by Morita~\cite{Morita}.  
{So far, there have been several studies on
variants of the Johnson filtrations and homomorphisms for mapping
class groups and other groups, {including the works}
\cite{AK,AN,Cooper,Kaneko,Levine1,Levine2,McNeill,MT,NT,Paris,Perron_0,Perron_p},
where the lower central series are replaced with some other descending
series.

The purpose of this paper is to generalize the Johnson filtrations and
homomorphisms to an arbitrary group acting on another group with a
descending series called an \emph{extended N-series}.  Our
constructions do not only give a generalized setting in order to view
the above-mentioned variants from a unified viewpoint, but also
provide new variants of the Johnson filtration and homomorphisms for
the mapping class group of a handlebody.}

\subsection{Extended N-series and extended graded Lie algebras}
\label{sec:n-series-extended}

An \emph{N-series} $K_+=(K_i)_{i\ge 1}$ of a group~$K$, introduced by Lazard \cite{Lazard}, is a  descending  series
\begin{gather*}
  K= K_1\ge  K_2\ge  \cdots
\end{gather*}
such that  $[K_i,K_j] \le  K_{i+j}$ for all $i,j\ge 1$.  The most
familiar example of an N-series is the lower central series
$\Gamma _+K=(\Gamma _iK)_{i\ge 1}$ defined inductively by $\Gamma _1K=K$
and $\Gamma _{i+1}K=[K,\Gamma _iK]$ for $i\ge 1$.
It  is the smallest N-series of $K$, i.e., we have
$\Gamma _iK \le  K_i$ for all $i\ge 1$ and for all N-series $(K_i)_{i\ge 1}$ of $K$.

By a \emph{graded Lie algebra} we mean a Lie algebra
$L_+=\bigoplus_{i\ge 1}L_i$ over $\modZ $ such that $[L_i,L_j]\subset L_{i+j}$ for
$i,j\ge 1$.
To every N-series $K_+$ is associated a graded Lie algebra
$$\gr_+(K_+)=\bigoplus_{i\ge 1}K_i/K_{i+1},$$ where the Lie bracket is
induced by the commutator operation.

An \emph{extended N-series}, studied  in this paper, is a natural
generalization of  N-series.  An extended N-series
$K_*=(K_i)_{i\ge 0}$ of a group $K$ is a  descending  series
\begin{gather}
  \label{e2}
  K= K_0\ge K_1\ge  K_2\ge  \cdots
\end{gather}
such that  $[K_i,K_j] \le  K_{i+j}$ for all $i,j\ge 0$.
Alternatively, a  descending  series \eqref{e2} is an extended N-series if the
positive part $K_+=(K_i)_{i\ge 1}$ is an N-series and if $K_i$ is a normal subgroup of $K$ for all $i\ge 1$.
Note that an N-series $K_+$ canonically extends to an
extended N-series by setting $K_0=K_1$.

An \emph{extended graded Lie algebra} (abbreviated as \emph{eg-Lie algebra}) $L_\bu=(L_i)_{i\ge 0}$ is a pair of
a graded Lie algebra ${L_+=\bigoplus_{i\ge 1}L_i}$ and a group $L_0$ acting on $L_+$.
To  each extended N-series $K_*$, we associate an eg-Lie
algebra $\gr_\bu(K_*)=(\gr_i(K_*))_{i\ge 0}$, consisting of
the graded Lie algebra
  $\gr_+(K_*)=\gr_+(K_+)$ associated to the N-series part $K_+$ of
  $K_*$, and the action of $\gr_0(K_*)=K_0/K_1$ on $\gr_+(K_+)$ induced
  by conjugation.

\subsection{Johnson filtrations and Johnson homomorphisms}
\label{sec:johns-filtr-mapp}

To recall Johnson's approach to mapping class groups,
assume  that $\Sigma $ is a compact, connected, oriented surface  with $\partial \Sigma \cong S^1$.
Let $K = \pi_1(\Sigma,\star)$, where $\star \in \partial \Sigma$,
and let $G$ be the mapping class group of $\Sigma$ relative to $\partial \Sigma$.
The natural action of $G$ on~$K$ gives rise to the Dehn--Nielsen representation
$$
\rho:G \longrightarrow \Aut(K).
$$

Let $K_+=\Gamma _+K$ be the lower central series of $K$.
The \emph{Johnson filtration} $G_*=(G_m)_{m\ge 0}$ of $G$ is defined by
\begin{gather*}
  G_m = \ker(\rho_m: G \lto \Aut(K/K_{m+1})),
\end{gather*}
where $\rho _m(g)(kK_{m+1})=\rho (g)(k)K_{m+1}$.
The series $G_*$ is an extended N-series.
The subgroup $G_1$ is known as the \emph{Torelli  group} of $\Sigma$,
and it is well known that  $\bigcap_{m\ge 0}G_m=\{1\}$.

For  $m\geq 1$, the \emph{{$m$th} Johnson homomorphism}
$$
\tau_m: G_m \longrightarrow \Hom (K_1/K_2, K_{m+1}/K_{m+2} ),
$$
is defined by
\begin{gather*}
  \tau _m(g)(kK_2)=g(k)k^{-1}K_{m+2}\quad \text{for $g\in G_m$, $k\in K_1$}.
\end{gather*}
Thus, $\tau _m$ measures the extent to which the action of $G_m$ on
$K/K_{m+2}$ fails to be trivial; in particular,
$\ker(\tau_m)=G_{m+1}$.  We can identify
$\Hom(K_1/K_2,K_{m+1}/K_{m+2})$ with the group $\Der_m(\gr_+ (K))$ of
degree $m$ derivations of $\gr_+(K)$, since the associated graded Lie
algebra $\gr_+(K)=\bigoplus_{m\ge 1}K_m/K_{m+1}$ is free on its degree
$1$ part~$K_1/K_2$.  Thus the $\tau_m$'s for $m\ge 1$ induce
homomorphisms
\begin{gather*}
  \bar\tau_m: G_m/G_{{m+1}} \longrightarrow \Der_m(\gr_+(K)),
\end{gather*}
forming  an injective morphism of graded Lie algebras
\begin{gather*}
  \bar\tau_+: \gr_+(G) \longrightarrow \Der_+(\gr_+(K)),
\end{gather*}
where $\Der_+(\gr_+(K))=\bigoplus_{m\ge 1}\Der_m(\gr_+(K))$ is the Lie
algebra of positive-degree derivations of $\gr_+(K)$. {This morphism of graded Lie algebras,
which contains all the Johnson homomorphisms, {was}
introduced by Morita \cite[Theorem 4.8]{Morita};
we call it the \emph{Johnson morphism}.}
From an algebraic viewpoint, it is important to determine
the image of $\bar \tau_+$, which is a Lie subalgebra of $\Der_+   (\gr_+(K)   )$.
We refer the reader to Satoh's survey \cite{Satoh_surface} for further details and references.

{We can extend} $\Der_+(\gr_+(K))$  to an eg-Lie algebra $\Der_\bu(\gr_+(K))$,
where  {the group} $\Der_0(\gr_+(K)) = \Aut(\gr_+(K))$ acts on $\Der_+(\gr_+(K))$ by conjugation.
Then the map $\bar\tau _+$ naturally extends to a morphism of eg-Lie algebras
\begin{gather}
  \label{e3}
  \bar\tau _\bu:\gr_\bu(G)\lto\Der_\bu(\gr_+(K)),
\end{gather}
whose  degree $0$ part
\begin{gather*}
  \bar\tau _0:\gr_0(G)=G_0/G_1\lto\Der_0(\gr_+(K))\simeq\Aut(H_1(\Sigma;\Z))
\end{gather*}
is given by the natural action of the mapping class group on homology.

\subsection{The Johnson morphisms associated to extended N-series actions}\label{sec:aaa}

We develop a theory of Johnson homomorphisms in the general situation
where an extended N-series $G_*=(G_m)_{m\ge0}$ of a group $G$
\emph{acts} on an extended N-series $K_*=(K_m)_{m\ge0}$ of another
group $K$.  This means that a left action
\begin{gather*}
  G\times K\longrightarrow K,\quad (g,k)\longmapsto   g(k),
\end{gather*}
of $G$ on $K$  satisfies
\begin{equation} \label{[G,K]}
  g(k)k^{-1}\in K_{i+j}  \quad \text{for all $g\in G_i$, $i\ge 0$
    and $k\in K_j$, $j\ge 0$.}
\end{equation}

We say that a group $G$  \emph{acts} on an extended N-series
$K_*$ if  $g(K_j)=K_j$ for all~$j\ge 0$.  In this case, we have
an extended N-series $\modF _*^{K_*}(G)$ of $G$ acting on $K_*$,
defined by
\begin{gather}
  \label{e6}
  \modF _i^{K_*}(G)=\{g\in G\;|\;g(k)k^{-1}\in K_{i+j}\ \text{for all $k\in K_j$, $j\ge 0$}\}.
\end{gather}
 We call $\modF _*^{K_*}(G)$ the \emph{Johnson filtration} of $G$ induced by $K_*$.

To each extended graded Lie algebra $L_\bu$, we associate the \emph{derivation eg-Lie algebra} $\Der_\bu(L_\bu)$ (see Theorem \ref{r101}).
The degree $0$ part $\Der_0(L_\bu)$ is the automorphism
group $\Aut(L_\bu)$ of $L_\bu$; the positive part $\Der_+(L_\bu)$ is
the Lie algebra of positive-degree derivations of
$L_\bu$.  Here, for
$m\ge 1$, a \emph{degree $m$ derivation} of $L_\bu$ consists of a
degree $m$ derivation $d_+$ of $L_+$ and a $1$-cocycle $d_0:L_0 \to
L_m$ satisfying certain compatibility condition (see Definition \ref{r1}).

To each action of an extended N-series $G_*$ on an extended
  N-series $K_*$, we associate
a morphism of extended graded Lie algebras
\begin{equation}   \label{j_general}
\bar \tau_\bu: \gr_\bu(G_*) \longrightarrow \Der_\bu ( \gr_\bu (K_*)),
\end{equation}
which we call the \emph{Johnson morphism},
and which generalizes \eqref{e3}.  The morphism $\bar \tau_\bu$ is
injective if and only if $G_*$ is the Johnson filtration induced by $K_*$.
(See Theorem~\ref{r87}.)

\subsection{The case of N-series}\label{sec:case-n-series}

If $K_*$ is the extension of an N-series $K_+=(K_m)_{m\ge 1}$,
{then the previous constructions specialize as follows.}
The target $\Der_\bu(\gr_\bu(K_*))=\Der_\bu(\gr_+(K_+))$ of the
Johnson morphism~\eqref{j_general} consists of the
  automorphism group $\Der_0(\gr_+(K_+))=\Aut(\gr_+(K_+))$ of the
  graded Lie algebra $\gr_+(K_+)$ and the graded Lie algebra
  $\Der_+(\gr_+(K_+))$ of positive-degree derivations of $\gr_+(K_+)$.

These simplifications recover the usual Johnson homomorphisms
\cite{Johnson_survey,Morita} and Andreadakis' constructions
\cite{Andreadakis} since, if $K_+=\Gamma _+K$ is the lower central series of a free group $K$,
then $\Der_+(\gr_+(K_+))$ is isomorphic to the Lie
algebra of ``truncated derivations''
$$
D_+ (\gr_+(K_+)) :=
\bigoplus_{m\geq 1} \Hom(K_1/K_2, K_{m+1}/K_{m+2}).
$$

We also consider the rational lower central series, and two mod-$p$ versions of the lower
  central series for a prime $p$. When $K=\pi_1(\Sigma)$ for a
  surface $\Sigma $, we recover the
``mod-$p$ Johnson homomorphisms''  introduced by {Paris \cite{Paris},}
Perron \cite{Perron_p} and Cooper~\cite{Cooper},
which are suitable for
the study of the \emph{mod-$p$ Torelli group}.
It is the subgroup of the mapping
class group consisting of elements acting trivially on $H_1(\Sigma ;\Z/p\Z)$.

{
After the first version of this manuscript {was}
released, {the authors were}
informed that Darn\'e, in his Ph.D. thesis {in preparation} \cite{Darne}, {constructed}
the same generalization of the Johnson morphism for an arbitrary $N$-series acting on another $N$-series.}

\subsection{Extended N-series associated to pairs of groups}
\label{sec:extended-n-series-2}

We introduce  two other types of extended
N-series $K_*$, each  associated with a pair $(K,N)$
of a group~$K$ and a normal subgroup $N$.

First, we associate to $(K,N)$ an extended N-series $K_*$
defined by $K_0=K$ and $K_m=\Gamma _mN$ for $m\ge1$.
An important case is where $N$ is free; this happens in particular
when $K$ is free. In this case, the positive part $\gr_+(K_+)$ of
the associated eg-Lie algebra $\gr_\bu(K_*)$ is a free
Lie algebra on its degree $1$ part $K_1/K_2=N/\Gamma _2 N$.  Unlike the classical
case where $K_0=K_1$, we have a non-trivial action of $K_0/K_1=K/N$ on
$\gr_+(K_+)$.  
This situation arises when we consider the action of
the mapping class group of a handlebody $V_g$ of genus $g$ (based with
a disc in the boundary) on  $\pi _1(V_g)$.
In fact, our study of generalized Johnson   homomorphisms for extended N-series
arises from the study of this action of the handlebody mapping class group.
We remark here that our generalized Johnson homomorphisms
{determine}
  McNeill's    ``higher order Johnson homomorphisms'' \cite{McNeill} on some subgroups of the
    surface mapping class group{, when $N$ is any characteristic subgroup of the fundamental group $K$ of a surface.}

Second, we associate to a pair $(K,N)$ with $[K,K]\le N$ the smallest
  extended N-series $K_*$ such that $K_0=K_1=K$ and $K_2=N$.
An example is the  ``weight filtration'' of  $K=\pi_1(\Sigma)$
for a punctured  surface $\Sigma $; thus, we
recover the generalizations of the Johnson homomorphisms on the  mapping class group of $\Sigma $
studied by Asada and Nakamura \cite{AN}.  In a different direction, we
obtain a new notion of Johnson homomorphisms on the ``Lagrangian'' mapping class group of a surface
studied from the point of view of finite-type invariants by
Levine, who also proposed a related notion of Johnson homomorphisms
\cite{Levine1,Levine2}.  This will be studied in the Ph.D. thesis of
Vera in connection with the ``tree reduction'' of the LMO
functor~$\widetilde{Z}$ introduced in \cite{CHM}.

\subsection{Formality of extended N-series}

We show that an action of an N-series~$G_+$ of a group $G$ on an extended N-series $K_*$ of a group $K$
has an ``infinitesimal'' counterpart if $K_*$ is formal in the following sense.

The extended N-series $K_*$ induces a
filtration on the group algebra $\Q[K]$. We say that $K_*$ is
\emph{formal} if the completion of $\Q[K]$ with respect to this
filtration is isomorphic to the degree-completion of the associated
graded of $\Q[K]$ through an isomorphism which is the identity on
the associated graded.  By generalizing Quillen's result
for the lower central series \cite{Quillen}, we show that the
associated graded of $\Q[K]$ is canonically isomorphic to the
``universal enveloping algebra'' of the eg-Lie $\Q$-algebra $\gr_\bu^\Q(K_*)$
(see  Theorem~\ref{Quillen}).
(Here $\gr_\bu^\Q(K_*)$ is given by $K_0/K_1$
in degree~$0$ and by $(K_m/K_{m+1}) \otimes \Q$ in degree $m\geq 1$.)
We can thus characterize the formality of $K_*$ in terms of
``expansions'' of   $K$, generalizing the Magnus expansions
for free groups. Then, we prove that such an expansion $\theta$
induces a filtration-preserving map
$$
\varrho^\theta: G \longrightarrow \prod_{m\geq 1} \Der_m(\gr_\bu^\Q(K_*)),
$$
which induces
$$
\bar \tau_+^\Q: \gr_+(G_*) \longrightarrow \Der_+(\gr_\bu^\Q(K_*)),
$$ the positive part of the rational version $\bar\tau_\bu^{\Q}$ of
$\bar\tau_\bu$ in \eqref{j_general}  (see Theorem \ref{infinitesimal_action}).
Thus, we may regard the map $\varrho^\theta$  as an ``infinitesimal version'' of the action
$$G_+\lto \Aut(K_*),$$
containing all the generalized Johnson homomorphisms with coefficients in $\Q$.

\subsection{Organization of the paper}\label{sec:organization-paper}

We organize the rest of the paper as follows. In Section \ref{sec:preliminaries},
 we fix some notations about groups. Sections
\ref{sec:extended-n-series} and \ref{sec:assoc-grad-lie}
deal with extended N-series and extended graded Lie algebras, respectively.
In Section~\ref{sec:derivation-eg-lie-1}, we introduce the extended graded Lie algebra consisting of the derivations of an extended graded Lie algebra.
In Section~\ref{sec:johns-homom-an}, we construct and study the
Johnson morphism induced by an extended N-series action.
In Section~\ref{sec:trunc-deriv-eg-1}, we consider truncations of the derivations of an extended graded Lie algebra.
In Section~\ref{sec:en-series-associated-1}, we specialize our constructions to N-series
and, in  Section \ref{sec:lower-central-series-1}, we illustrate these with variants of the lower central series
in order to recover several versions of the Johnson homomorphisms  in the literature.
In Section~\ref{sec:en-series-associated-2}, we consider two types of extended N-series defined by a pair of groups,
and we announce some works in progress.
Section~\ref{sec:filtr-group-rings-1} computes the associated graded
of the filtration of a group algebra induced by an extended N-series.
We consider the case of formal extended N-series in Section~\ref{sec:formality-extended-n}.

\begin{acknowledgments}
The work of K.H. is partly supported by JSPS KAKENHI Grant Number 15K04873.
\end{acknowledgments}

\section{Preliminaries in group theory}   \label{sec:preliminaries}

Here we recall a few facts about groups and fix some notations.

\subsection{Groups} \label{sec:gener-comm-calc}

Let $G$ be a group.  By $N \leq G$ we mean that $N$ is a subgroup of~$G$,
and by $N\trll G$ that $N$ is a normal subgroup of $G$.  Given
a subset $S$ of $G$, let $\langle S \rangle$ denote the subgroup of
$G$ generated by $S$, and $\lala S \rara =\lala S \rara_G$ the
normal subgroup in $G$ generated by $S$.

For $g,h\in G$, set
\begin{gather*}
  [g,h] = ghg^{-1}h^{-1},\qquad   {}^g h = ghg^{-1}, \qquad h^g = g^{-1}hg.
\end{gather*}
We will freely use the following commutator identities:
\begin{gather}
  \label{ee7} [a,bc] = [a,b]\cdot {}^b[a,c],\quad\quad
	[ab,c]={}^a[b,c]\cdot [a,c],\\
  \label{ee8} {[a,b^{-1}]^{-1} =  [a,b]^{{b}}},\quad\quad
	{[a^{-1},b]^{-1} =  [a,b]^{a}},\\
  \label{ee54} [[a,b],{}^bc]\cdot [[b,c],{}^ca]\cdot [[c,a],{}^ab]=1.
\end{gather}

We will need the well-known three subgroups lemma:

\begin{lemma}
  \label{3s}
  If $A,B,C\le G$, $N\trll G$, $[A,[B,C]]\le N$ and $[B,[C,A]]\le N$, then we
  have $[C,[A,B]]\le N$.
\end{lemma}

\subsection{Group actions}  \label{sec:group-acti-semid}

Consider an action of a group $G$ on a group $K$:
$$
G \times K \longrightarrow K, \quad (g,k) \longmapsto g(k).
$$
Let $K \rtimes G$ denote  the semidirect product of $G$ and $K $,
which is the set $K \times G$ with multiplication
\begin{gather*}
  (k,g)\, (k',g')=(kg(k'),gg').
\end{gather*}
We naturally regard
$K $ and $G$ as subgroups of $K \rtimes G$.
Then, for $g\in G$, $k\in K $,
$$
{}^g k = gkg^{-1} = g(k)  \  \in K \le   K\rtimes G
$$
and
\begin{gather*}
  [g,k]=g k g^{-1} k^{-1}=g(k)k^{-1} \ \in K \le   K\rtimes G.
\end{gather*}
We will use these notations whenever a group $G$ acts on another group $K$.

For  $G' \leq G$ and $K' \leq K $,
let $[G',K ']$ denote the subgroup of $K $ generated by the
elements $[g',k']$ for $g'\in G'$, $k'\in K '$, and let
${}^{G'}\!K '$ denote the subgroup of $K $ generated by the
elements ${}^{g'}\!k' $ for $g'\in G'$, $k'\in K '$.
For  $g\in G$, let $[g,K']$ denote the set of elements of $K$ of the form $[g,k']$ for all $k'\in K'$.

\section{Extended N-series and the Johnson filtration}  \label{sec:extended-n-series}

In this section, we introduce the notion of extended N-series and
  the Johnson filtration for an action of a group on an extended N-series.

\subsection{N-series}  \label{sec:n-series}

An \emph{N-series} \cite{Lazard} of a group $G$ is a  descending  series
$$
G=G_1 \geq G_2\ge \cdots \geq G_i  \geq \cdots
$$
such that
\begin{gather}
  \label{e56}
  [G_m,G_n]  \le   G_{m+n}\quad  \text{for $m,n\ge 1$}.
\end{gather}
Note that $(G_i)_{i\ge 1}$ is a central series, i.e., $[G,G_i] \le  G_{i+1}$ for $i\geq 1$.
In particular, we have $G_i\trll G$ for $i\geq 1$.

As mentioned in the introduction, the lower central series of $G$
  is the smallest N-series of $G$.

\subsection{Extended N-series}  \label{sec:extended-n-series-1}

An \emph{extended N-series} $G_*=(G_m)_{m\ge 0}$ is a  descending  series
\begin{gather*}
G_0\ge G_1\ge  \cdots \ge G_k \geq \cdots
\end{gather*}
 such that
\begin{gather}
  \label{e59}
  [G_m,G_n]  \le  G_{m+n}\quad  \text{for $m,n\ge 0$}.
\end{gather}

For every extended N-series $G_*=(G_m)_{m\ge0}$, the subseries
$G_+ = (G_m)_{m\ge 1}$  is an $N$-series.  Conversely, every N-series $(G_m)_{m\geq 1}$  extends to
an extended N-series by setting $G_0=G_1$.

A \emph{morphism} $f: G_*\rightarrow K_*$ between extended N-series $G_*$ and $K_*$
is a  homomorphism $f: G_0\rightarrow K_0$ such that  $f(G_m)\subset K_m$ for all
$m\ge 0$.  Let $\eNs$ denote the category of extended N-series  and morphisms.

In the rest of this section, we adapt several usual constructions for groups to extended N-series.

\subsection{Actions on extended N-series}   \label{sec:en-series-actions}

Let $K_*$ be an extended N-series.
By an \emph{action} of an extended N-series $G_*$ on $K_*$, we mean an
action of  $G_0$ on $K_0$ such that
\begin{gather}
  \label{e60}
  [G_m,K_n]\subset K_{m+n}\quad \text{for $m,n\ge 0$}.
\end{gather}
By an \emph{action} of a group $G$ on $K_*$, we mean an action of
$G$ on $K_0$ such that
\begin{gather}
  \label{e7}
  g(K_n)=K_n\quad \text{for $g\in G$, $n\ge 0$}.
\end{gather}
Note that if $G_*$ acts on $K_*$, then $G_0$ acts on $K_*$.

\subsection{Johnson filtrations} \label{sec:en-series-induced}

If a group $G$ acts on an extended N-series $K_*$, then we have an
  extended N-series $\modF _*^{K_*}(G)$ of $G$ defined by
\begin{gather}
  \label{e11}
  \modF ^{K_*}_m(G)
  =\{g\in G\zzzvert [g,K_n]\subset K_{m+n} \text{ for $n\ge 0$}\}\quad
\end{gather}
for every $m\ge 0$, which we call the \emph{Johnson filtration} of $G$ induced by $K_*$.

\begin{proposition}
  \label{r77}
  If a group $G$ acts on an extended N-series ${K_*}$, then the Johnson
  filtration  $\modF _*^{K_*}(G)$ is the largest extended N-series of $G$ acting on $K_*$.
\end{proposition}

\begin{proof}
  Set $G_*=\modF _*^{K_*}(G)$.
  One easily checks that  $G_*$
  is a descending series of $G$,
  and that $[G_m,K_n]\subset K_{m+n}$ for $m,n\ge 0$.
  We have    $[G _m,G _n]\subset G _{m+n}$ for $m,n\ge 0$, since for $i\ge 0$
  \begin{gather*}
    \begin{split}
      [ [G _m,G _n],K_i]
      &\subset \lala\, [G _m,[G _n,K_i]]\cdot [G _n,[G _m,K_i]]\,
      \rara_{K_0\rt G}\quad \text{({by} Lemma \ref{3s})} \\
      &\subset \lala\, [G _m,K_{n+i}]\cdot [G _n,K_{m+i}]\, \rara_{K_0\rt G} \\
      &\subset \lala\, K_{m+n+i}\, \rara_{K_0\rt G} = K_{m+n+i}.
    \end{split}
  \end{gather*}
  Hence $G_*$ is an extended N-series acting on $K_*$.
  It is clear from the definition of~$G_*$ that, if $G'_*$ is
    another extended N-series of $G$ acting on $K_*$, then $G'_m\le G_m$.
\end{proof}

\begin{remark}
  \label{r13}
  In the proof of Proposition \ref{r77}, we {did not use}
  the condition $[K_m,K_n]\le K_{m+n}$, $m,n\ge0$.  Therefore, we can generalize
  Proposition \ref{r77} {to any normal series $K_*=(K_m)_{m\ge0}$ of a group $K$.}
\end{remark}

\subsection{Automorphism group of an extended N-series} \label{sec:autom-en-seri}

Let $K_*$ be an extended N-series.
Define the \emph{automorphism group of $K_*$ by}
\begin{eqnarray}
  \label{e90}  \Aut(K_*)
  &=&
 \{g\in \Aut(K_0)\zzzvert g(K_i) = K_i \text{ for $i\ge0$} \} ,
\end{eqnarray}
which is the largest subgroup of $\Aut(K_0)$ acting on $K_*$.  Note
that a  homomorphism $ G\rightarrow \Aut(K_*)$ is
equivalent  to an action of $G$ on $K_*$.

Let $\Aut_*(K_*)$ denote the Johnson filtration
  $\modF _*^{K_*}(\Aut(K_*))$ of $\Aut(K_*)$ induced by $K_*$;
  thus,
  \begin{equation} \label{A}
    \Aut_m(K_*)=\{g\in \Aut(K_*)\zzzvert [g,K_n]\subset K_{m+n}\text{ for $n\ge 0$}\}
  \end{equation}
  for $m\ge 0$.
Note that a morphism {$ G_*\rightarrow \Aut_*(K_*)$ of extended N-series}
is {equivalent} to an action of $G_*$ on~$K_*$.
The following lemma is easily {verified}.

\begin{lemma}
  \label{r96}
  Let $K_*$ be an extended $N$-series.
  \begin{enumerate}
  \item If $K_m$ is characteristic in $K_0$ for {all}
    $m\ge 1$, then $\Aut(K_*)=\Aut(K_0)$.
  \item If $K_m$ is characteristic in $K_1$ for {all}
    $m\ge 2$, then $\Aut(K_*)=\Aut(K_0,K_1)$,
    where $\Aut(K_0,K_1) = \{g\in \Aut(K_0)\zzzvert g(K_1)=K_1\}$.
  \end{enumerate}
\end{lemma}

\begin{example}
  \label{r17}
  Let $K_*$ be an extended N-series.  Then $K_*$ acts on itself via
  the conjugation $K\times K\to K$, $(k,k')\mt {}^kk'$.
  {Thus}, we have
  a morphism of extended N-series
  \begin{gather*}
    \Ad^{K_*}: K_*\lto \Aut_*(K_*),
  \end{gather*}
  called the {\em adjoint action} of $K_*$.
  In general, $K_*$ does
  not coincide with the Johnson filtration $\mathcal{F}^{K_*}_*(K_0)$
  of $K_0$ induced by its action on $K_*$.  For example, if $K_0$ is
  abelian, then $\mathcal{F}^{K_*}_*(K_0)=(K_0)_{n\ge0}$, which is
  different from $K_*$ in general.  See Remark \ref{r9} for an
  example where {we have} $K_*=\mathcal{F}^{K_*}_*(K_0)$.
\end{example}

\section{Extended graded Lie algebras}  \label{sec:assoc-grad-lie}

It is well known \cite{Lazard} that
to each N-series is associated a graded Lie algebra over~$\Z$.
Here we associate to each extended N-series an eg-Lie algebra.

\subsection{Graded Lie algebras}  \label{sec:graded-lie-algebras}

Recall that a \emph{graded Lie algebra} $L_+=(L_m)_{m\ge 1}$ consists of abelian
groups $L_m$, $m\ge 1$, and bilinear maps
\begin{gather*}
  [{\cdot,\cdot}]: L_m\times L_n\rightarrow L_{m+n}
\end{gather*}
for $m,n\ge 1$ such that
\begin{itemize}
\item $[x,x]=0$ for  $x\in L_m$, $m\ge 1$,
\item $[x,y]+[y,x]=0$ for  $x\in L_m$, $y\in L_n$, $m,n\ge 1$,
\item $[x,[y,z]]+[y,[z,x]]+[z,[x,y]]=0$ for  $x\in L_m$, $y\in L_n$,
  $z\in L_p$, $m,n,p\ge 1$.
\end{itemize}
Also, let $L_+$  denote the direct sum $\bigoplus_{m\ge 1}L_m$ by abuse of notation.

A \emph{morphism} $f_+: L_+\rightarrow L'_+$ of graded Lie algebras is a family
$f_+=(f_i)_{i\ge 1}$ of  homomorphisms $f_i: L_i\rightarrow L'_i$
such that $f_{i+j}([x,y])=[f_i(x),f_j(y)]$ for all $x\in L_i$, $y\in L_j$, $i,j\ge 1$.
An \emph{automorphism} of $L_+$ is an invertible morphism from $L_+$ to
itself.  Let $\Aut(L_+)$ denote the group of automorphisms of $L_+$.

An \emph{action} of a group $G$ on $L_+$ is a homomorphism from $G$ to
$\Aut(L_+)$.  In other words, {it} is a degree-preserving action $(g,x)\mapsto {}^g x$ of $G$
on  $L_+$ such that
\begin{gather}
  \label{e94}
	{}^g[x,y]=[{}^gx,{}^gy] \quad \hbox{{for $g\in G$ {and} $x,y\in L_+$}.}
\end{gather}

\subsection{Extended graded Lie algebras} \label{sec:extended-graded-lie}

An \emph{extended graded Lie algebra} (abbreviated as \emph{eg-Lie algebra})
$L_\bu=(L_m)_{m\ge 0}$ consists of
\begin{itemize}
\item a group $L_0$,
\item a graded Lie algebra $L_+  = (L_m)_{m\ge 1}$,
\item an action {$(g,x)\mapsto {}^gx$}
of $L_0$ on $L_+$.
\end{itemize}

A \emph{morphism} $f_\bullet=(f_m: L_m\rightarrow L'_m)_{m\ge 0}: L_\bu\rightarrow L'_\bu$ between eg-Lie algebras $L_\bu$
and $L'_\bu$ consists of
\begin{itemize}
\item a  homomorphism $f_0: L_0\rightarrow L'_0$,
\item a {graded Lie algebra morphism}
  $f_+  = (f_m)_{m\ge 1}: L_+\rightarrow L'_+$,
\end{itemize}
such that
$$
f_m({}^xy)={}^{f_0(x)} (f_m(y) )
$$
for all $x\in L_0$, $y\in L_m$, $m\ge 1$.
Let $\egL$ denote the category of eg-Lie algebras and morphisms.

\subsection{From extended N-series to eg-Lie algebras}\label{sec:from-extended-n}

For each {extended N-series}~$K_*$, {we} define the
\emph{associated eg-Lie algebra}
$\bKb = \gr_\bu(K_*)$ {as follows.} {Set}
\begin{gather*}
  \bK_m=\gr_m(K_*) =K_m/K_{m+1}
\end{gather*}
for all $m\geq 0$.
The group $\bK_0$ is not abelian in general,
whereas $\bK_m$ is abelian for $m\geq 1$.
Thus we will use multiplicative notation for the former, and the additive notation for the latter.
The Lie bracket $[\cdot,\cdot]: \bK_m\times \bK_n\rightarrow \bK_{m+n}$ in $\bKb$
 is given by
\begin{gather}
  \label{e95}
  [aK_{m+1},bK_{n+1}]=[a,b]K_{m+n+1}
\end{gather}
{for  $m,n\ge 1$, and}  the action of $\bK_0$ on $\bK_m$ is given by
\begin{gather}
  \label{e96}
  {}^{(aK_1)}(bK_{m+1})=({}^ab)K_{m+1}.
\end{gather}
Observe that $\bar{K}_+$ is the usual   graded Lie algebra associated to the N-series $K_+$ (see \cite[Theorem  2.1]{Lazard}).

There is a functor $ \grb: \eNs\rightarrow \egL$.
Indeed, every morphism $f: G_*\rightarrow K_*$ in $\eNs$
induces a morphism $ \gr_\bu(f): \grb(G_*)\rightarrow \grb(K_*)$ in $\egL$ defined by
\begin{gather}
  \label{e97}
  \grb(f)(gG_{m+1})=f(g)K_{m+1}, \quad {(g\in G_m, m\ge 0)}.
\end{gather}

\section{Derivation eg-Lie algebras of eg-Lie algebras}\label{sec:derivation-eg-lie-1}

In this section, we introduce the derivation eg-Lie algebra of an
 eg-Lie algebra, which generalizes the derivation Lie algebra of a  graded Lie algebra.

\subsection{Derivations of an eg-Lie algebra}
\label{sec:derivations-of-eg-lie}

Let $L_\bu$ be an eg-Lie algebra.

\begin{definition}
  \label{r1}
  Let $m\ge 1$.  A \emph{derivation}
  $d=(d_i)_{i\ge 0}$ of $L_\bu$ of degree $m$ is a family of maps
  $d_i: L_i\rightarrow L_{m+i}$ satisfying the following  conditions.
\begin{enumerate}
\item $d_+ = (d_i)_{i\ge 1}$ is a derivation of the graded Lie
  algebra $L_+$, i.e., the
  $d_i$ for $i\ge 1$ are homomorphisms such that
$$
  d_{i+j}([a,b])=[d_i(a),b]+[a,d_j(b)]
$$
for $a\in L_i$, $b\in L_j$, $i,j\ge 1$.
\item The map $d_0: L_0\rightarrow L_m$ is a $1$-cocycle.
  In other words, we have
$$
  d_0(ab)=d_0(a)+{}^a(d_0(b))
$$
  for $a,b\in L_0$.
\item We have
$$
    d_i({}^ab)=[d_0(a),{}^ab]+{}^a(d_i(b))
$$
 for $a\in L_0$, $b\in L_i$, $i\ge 1$.
\end{enumerate}
\end{definition}

For $m\ge 1$, let $\Der_m(L_\bu)$ be the  group of derivations of
$L_\bu$ of degree $m$.  Set $\Der_+(L_\bu) =
(\Der_m(L_\bu))_{m\ge 1}$.

\begin{theorem}
  \label{r2}
  We have a graded Lie algebra structure on $\Der_+(L_\bu)$  such that, for $m,n\ge 1$, the Lie bracket
\begin{gather*}
  [\cdot,\cdot]: \Der_m(L_\bu)\times \Der_n(L_\bu) \longrightarrow \Der_{m+n}(L_\bu)
\end{gather*}
is given by
    \begin{gather}
      \label{e117}
      [d,d']_i(a)=
      \begin{cases}
	d_n(d'_0(a))-d'_m(d_0(a))-[d_0(a),d'_0(a)]&(i=0,a\in L_0),\\
	d_{n+i}(d'_i(a))-d'_{m+i}(d_i(a))&(i\ge 1,a\in L_i).
      \end{cases}
    \end{gather}
\end{theorem}

We call $\Der_+(L_\bu)$ the \emph{derivation graded Lie algebra} of $L_\bu$.

\begin{proof}[Proof of Theorem \ref{r2}]
For simplicity of notation,  set $D_+=\Der_+(L_\bu)$.

For $d\in D_m$, $d'\in D_n$, $m,n \geq 1$,
define $[d,d']=([d,d']_i: L_i \to L_{i+m+n})_{i\geq 0}$ by~\eqref{e117}.
We prove $[d,d'] \in D_{m+n}$ as follows.

First, {$[d,d']_+=([d,d']_i)_{i\geq 1}$ is a derivation of
  $L_+$} since the commutator of two derivations of a Lie
  algebra is a derivation.

Second, we verify that $[d,d']_0:L_0 \to L_{m+n}$ is a $1$-cocycle.  For  $a,b\in L_0$,
  \begin{gather*}
    \begin{split}
      & \ \quad [d,d'](ab)\\
      &=dd'(ab)-d'd(ab)-[d(ab),d'(ab)]\\
      &=d\big(d'(a)+{}^a(d'(b))\big)-d'\big(d(a)+{}^a(d(b))\big) -\big[d(a)+{}^a(d(b)),d'(a)+{}^a(d'(b))\big]\\
      &=dd'(a)+\big[d(a),{}^a(d'(b))\big]+{}^a(dd'(b))   -d'd(a)-\big[d'(a),{}^a(d(b))\big]-{}^a(d'd(b))\\
      &\quad -[d(a),d'(a)]- \big[d(a),{}^a(d'(b))\big]- \big[{}^a(d(b)),d'(a)\big]- \big[{}^a(d(b)),{}^a(d'(b))\big]\\
      &=dd'(a)+{}^a(dd'(b))-d'd(a) - {}^a(d'd(b))
      -[d(a),d'(a)]-{}^a[d(b),d'(b)]\\
      &=[d,d'](a)+{}^a\big([d,d'](b)\big).
    \end{split}
  \end{gather*}

  Third, for $a\in L_0$, $b\in L_i$, $i\ge 1$, we {have}
  \begin{gather*}
    \begin{split}
      [d,d']({}^ab)
      &=dd'({}^ab)-d'd({}^ab)\\
      &=      d\big([d'(a),{}^ab]+{}^a(d'(b))\big)
      -d'\big([d(a),{}^ab]+{}^a(d(b))\big)\\
      &=
      [dd'(a),{}^ab]+[d'(a),d({}^ab)]+[d(a),{}^a(d'(b))]+{}^a(dd'(b))\\
    &\quad -[d'd(a),{}^ab]-[d(a),d'({}^ab)]-[d'(a),{}^a(d(b))]-{}^a(d'd(b))\\
      &=
      [dd'(a),{}^ab]+\big[d'(a),[d(a),{}^ab]+{}^a(d(b))\big]+[d(a),{}^a(d'(b))]+{}^a(dd'(b))\\
      &\quad-[d'd(a),{}^ab]-\big[d(a),[d'(a),{}^ab]+{}^a(d'(b))\big]-[d'(a),{}^a(d(b))]-{}^a(d'd(b))\\ 
      &=
      [dd'(a),{}^ab]+[d'(a),[d(a),{}^ab]]+{}^a(dd'(b))\\
      &\quad-[d'd(a),{}^ab]-[d(a),[d'(a),{}^ab]]-{}^a(d'd(b))\\
      &=\big[dd'(a)-d'd(a)-[d(a),d'(a)],{}^ab\big]+{}^a\big(dd'(b)-d'd(b)\big)\\
    &=[ [d,d'](a),{}^ab]+{}^a\big([d,d'](b)\big).
    \end{split}
  \end{gather*}
  Therefore,  $[d,d']$ is a derivation of  $L_\bu$ of degree $m+n$.

  Now we show that the {maps} $[\cdot, \cdot]: D_m\times
  D_n \to D_{m+n}$ for $m,n\ge 1$ define a graded Lie algebra
  structure on $D_+$.  {Clearly, we have $[d,d]=0$
  and $[d,d']+[d',d]=0$ for $d,d'\in D_+$.} Thus it remains to check the Jacobi identity
  \begin{equation} \label{evaluated_Jacobi}
  [d,[d',d'']](a)+ [d'',[d,d']](a) + [d',[d'',d]](a) =0
  \end{equation}
  for $d,d',d''\in D_+$ and $a\in L_i$ with $i\geq 0$.
  For $i\geq1$, this is the standard fact that derivations of a Lie
  algebra {form}   a Lie algebra.  For $i=0$,  we have
  \begin{gather*}
    \begin{split}
      [d,[d',d'']](a)
      &=d[d',d''](a)-[d',d'']d(a)-\big[d(a),[d',d''](a)\big]\\
      &=d\big(d'd''(a)-d''d'(a)-[d'(a),d''(a)]\big)-\big(d'd''d(a)-d''d'd(a)\big)\\
      &\quad -\big[d(a),d'd''(a)-d''d'(a)-[d'(a),d''(a)]\big]\\
      &=dd'd''(a)-dd''d'(a)-[dd'(a),d''(a)]-[d'(a),dd''(a)]\\
      &\quad -d'd''d(a)+d''d'd(a)\\
      &\quad -[d(a),d'd''(a)]+[d(a),d''d'(a)] + [d(a),[d'(a),d''(a)]],
    \end{split}
  \end{gather*}
  from which {\eqref{evaluated_Jacobi} follows}.
\end{proof}

\subsection{Derivation eg-Lie algebras}  \label{sec:derivation-eg-lie}

Let $L_\bu$ be an eg-Lie algebra.

\begin{theorem}
  \label{r101}
  The derivation graded Lie algebra $\Der_+(L_\bu)$ extends to an
  eg-Lie algebra $\Der_\bu(L_\bu)$ by setting
  $\Der_0(L_\bu)   =\Aut(L_\bu)$ and by defining an action
\begin{gather}
  \label{e5}
  \Der_0(L_\bu)\times \Der_m(L_\bu) \longrightarrow \Der_m(L_\bu),\quad (f,d) \longmapsto {}^fd,
\end{gather}
for $m\ge 1$ by
\begin{gather}
  \label{e116}
  ({}^fd)_i(a)=f_{m+i}d_if_i^{-1}(a) \qquad {(i\ge 0, a\in L_i).}
\end{gather}
\end{theorem}

We call $\Der_\bu(L_\bu)=(\Der_m(L_\bu))_{m\ge 0}$ the \emph{derivation eg-Lie algebra} of $L_\bu$.

\begin{proof}
 {For simplicity of notation,  set $D_\bu=\Der_\bu(L_\bu)$.}
 {For $f\in D_0$, $d\in D_m$, $m\geq 1$, define
  ${}^fd=     (({}^fd)_i: L_i \to L_{i+m}     )_{i\geq 0}$
  by~\eqref{e116}.}  We prove  ${}^fd \in D_m$ as follows.

  First, we check {that} $({}^fd)_+$ is a derivation of  $L_+$.
  For  $a\in L_i$, $b\in L_j$, $i,j\ge 1$,
  \begin{gather*}
    \begin{split}
      ({}^fd)([a,b])
      &=fdf^{-1}([a,b])\\
      &=fd([f^{-1}(a),f^{-1}(b)])\\
      &=f\big([df^{-1}(a),f^{-1}(b)]+[f^{-1}(a),df^{-1}(b)]\big)\\
      &=[fdf^{-1}(a),b]+[a,fdf^{-1}(b)]\\
      &=\big[({}^fd)(a),b\big]+\big[a,({}^fd)(b)\big].
    \end{split}
  \end{gather*}

  Second, we check that $({}^fd)_0:L_0 \to L_m$ is a  $1$-cocycle. For $a,b\in L_0$,
  \begin{gather*}
    \begin{split}
      ({}^fd)(ab)
      &=fdf^{-1}(ab)\\
      &=fd\big(f^{-1}(a)f^{-1}(b)\big)\\
      &=f\Big(df^{-1}(a)+{}^{f^{-1}(a)}\big(df^{-1}(b)\big)\Big)\\
      &=fdf^{-1}(a)+{}^a(fdf^{-1}(b)) \, = \, ({}^fd)(a)+{}^a\big(({}^fd)(b)\big).
    \end{split}
    \end{gather*}

  Third, we {have}
  for $a\in L_0$, $b\in L_i$, $i\ge 1$,
  \begin{gather*}
    \begin{split}
      ({}^fd)({}^ab)
      &=fdf^{-1}({}^ab)\\
      &=fd({}^{f^{-1}(a)}(f^{-1}(b)))\\
      &=f\Bigl(\big[df^{-1}(a),{}^{f^{-1}(a)}(f^{-1}(b))\big]+{}^{f^{-1}(a)}\big(df^{-1}(b)\big)\Bigr)\\
      &=\big[fdf^{-1}(a),{}^ab\big]+{}^a\big(fdf^{-1}(b)\big)\\
      &=\big[({}^fd)(a),{}^ab\big]+{}^a\big(({}^fd)(b)\big).
    \end{split}
  \end{gather*}
  {Therefore,}   ${}^fd$ is a derivation of the eg-Lie algebra $L_\bu$  of degree $m$.

It is easy to check that the maps $D_0 \times D_m \to D_m$,
  $(f,d)\mt{}^fd$ for $m\ge 1$ form an action of $D_0$ on the graded abelian group $D_+$.
  Let us verify that this action preserves the Lie bracket of $D_+$.
  Let $g\in D_0$, $d\in D_m$, $d'\in D_n$ with $m,n\ge 1$, and let $a\in L_i$ with  $i\ge 0$. {For $i\geq1$,} we have
  \begin{gather*}
    \begin{split}
      ({}^g[d,d'])(a)
      &=g[d,d']g^{-1}(a)\\
      &=g(dd'-d'd)g^{-1}(a) =(gdg^{-1}gd'g^{-1}-gd'g^{-1}gdg^{-1})(a)=[{}^gd,{}^gd'](a)
    \end{split}
  \end{gather*}
  {and, for $i=0$, we have  \begin{gather*}
      \begin{split}
	({}^g[d,d'])(a)
	&=g[d,d']g^{-1}(a)\\
	&=gdd'g^{-1}(a) -gd'dg^{-1}(a) - g[dg^{-1}(a),d'g^{-1}(a)]\\
	&=gdg^{-1}gd'g^{-1}(a) -gd'g^{-1}gdg^{-1}(a) - [gdg^{-1}(a),gd'g^{-1}(a)]\\
	& =[{}^gd,{}^gd'](a).
      \end{split}
  \end{gather*}}
  {Hence}  $D_\bu$ is an eg-Lie algebra.
\end{proof}

\begin{example} 
  \label{r19}
Let $L_\bu$ be an eg-Lie algebra.  There is a morphism of eg-Lie
algebras
\begin{gather}
  \label{e12}
  \ad=\ad^{L_\bu}: L_\bu \lto \Der_\bu(L_\bu),
\end{gather}
called the {\em adjoint action} of $L_\bu$. {It is} defined by
\begin{gather*}
  \ad(a)(b) =
  \begin{cases}
 {}^ab & \text{for $a\in L_0$, $b\in L_n$, $n\ge0$},\\
 [a,b] & \text{for $a\in L_m$, $m\ge1$, $b\in L_n$, $n\ge0$},\\
  \end{cases}
\end{gather*}
where we set $[a,b]=a-{}^ba$ for $a\in L_m$, $m\ge1$ and $b\in L_0$.
The proof is straightforward and left to the reader.
\end{example}

\section{The Johnson homomorphisms of an extended N-series action}\label{sec:johns-homom-an}

In this section, we {generalize} Johnson homomorphisms {for an arbitrary action of extended  N-series $G_*$ on $K_*$.}
{These ``Johnson homomorphisms'' {form} a ``Johnson morphism''}
\begin{gather*}
  \bar{\tau}_\bu: \gr_\bu(G_*)\longrightarrow \Der_\bu(\gr_\bu(K_*))
\end{gather*}
{with values in the derivation eg-Lie algebra of $\gr_\bu(K_*)$.}

\subsection{Generalized Johnson homomorphisms}  \label{sec:gener-johns-homom-3}

In this subsection, we consider an extended N-series $G_*$ acting on an extended N-series $K_*$,
and we {set} $\bKb=\grb(K_*)$.
For {every} $m\geq 0$, we {will} define a homomorphism
$$
\tau _m  = \tau _m^{G_*,K_*}: G_m\longrightarrow \Der_m(\bKb),
$$
which we call the \emph{{$m$th} (generalized) Johnson homomorphism.}
{We treat the cases $m=0$ and $m>0$ separately.}

\begin{proposition}
  \label{r3}
There is a homomorphism $$\tau _0 : G_0 \longrightarrow
\Aut(\bKb)$$
which maps {each} $g\in G_0$ to  $\tau _0(g)=(\tau _0(g)_i:
\bK_i\rightarrow \bK_i)_{i\ge 0}$  defined by
        \begin{gather}
          \label{e62}
          \tau _0(g)_i\, (aK_{i+1})=  ({}^ga) K_{i+1}.
        \end{gather}
\end{proposition}

\begin{proof}
  {Let $\End(\bKb)$ denote} the monoid of endomorphisms of the eg-Lie algebra $\bKb$.
  Let $g\in G_0$. We prove that $\tau _0(g)\in \End(\bKb)$ is well  defined as follows.
{It is easy to see that $\tau _0(g)_i: \bK_i\rightarrow \bK_i$ is a
  well-defined homomorphism for $i\ge 0$.}

  Next, $(\tau _0(g)_i)_{i\ge 1}: \bar K_+ \to \bar K_+$ is a graded Lie algebra automorphism since,
  for $a\in K_i$, $b\in K_j$, $i,j\ge 1$, we have
  \begin{eqnarray*}
    \tau _0(g)([aK_{i+1},bK_{j+1}])
    &=&\tau _0(g)([a,b]K_{i+j+1})
    =({}^g[a,b])K_{i+j+1}
    =[{}^ga,{}^gb]K_{i+j+1}\\
    &=&\big[({}^ga)K_{i+1},({}^gb)K_{j+1}\big]
    =[\tau _0(g)(aK_{i+1}),\tau _0(g)(bK_{j+1})].
  \end{eqnarray*}

{We now check} the equivariance property:
  \begin{eqnarray*}
    \tau _0(g)\big({}^{(aK_1)}(bK_{i+1})\big)
    &=&\tau _0(g)\big(({}^ab)K_{i+1}\big)
    =({}^g({}^ab))K_{i+1}
    =\big({}^{({}^ga)}({}^gb)\big)K_{i+1}\\
    &=&{}^{({}^ga)K_1}\big(({}^gb)K_{i+1}\big)
    ={}^{\tau _0(g)(aK_1)}\big(\tau _0(g)(bK_{i+1})\big)
  \end{eqnarray*}
  for $a\in K_0$, $b\in K_i$, $i\geq 1$.
  Thus, we have $\tau _0(g)\in \End(\bKb)$.

  {The} map $\tau _0: G_0\rightarrow \End(\bKb)$ is a {monoid}
  homomorphism, i.e.,  {we have} $\tau _0(gg')=\tau _0(g)\tau _0(g')$ for   $g,g'\in G_0$.
  Indeed, for $a\in K_i$, $i\ge 0$, we have
  \begin{eqnarray*}
      \tau _0(gg')(aK_{i+1})&=&\big({}^{(gg')}a\big)K_{i+1}
      =\big({}^g({}^{g'}a)\big)K_{i+1}
      =\tau _0(g)\big(({}^{g'}a)K_{i+1}\big)\\
      &=&\tau _0(g)\big(\tau _0(g')(aK_{i+1})\big)
      =(\tau _0(g)\tau _0(g'))(aK_{i+1}).
  \end{eqnarray*}
  {Hence $\tau_0$ takes values in $\Aut(\bKb)$.}
\end{proof}

\begin{proposition}
  \label{r4}
For $m\geq 1$, there is a  homomorphism $${\tau _m : G_m \longrightarrow \Der_m(\bKb)}$$ which maps
{each} $g\in G_m$ to  $\tau _m(g)=(\tau _m(g)_i: \bK_i\rightarrow \bK_{m+i})_{i\ge 0}$ defined~by
        \begin{gather}
          \label{e63}
          \tau _m(g)_i\, (aK_{i+1})=[g,a]K_{m+i+1}.
        \end{gather}
\end{proposition}

 \begin{proof}
   Let $g\in G_m$. We show that   $\tau _m(g)\in \Der_m(\bKb)$ is well
   defined as follows.

  Since $G_*$ acts on  $K_*$, we easily see that
  the map $\tau _m(g)_i: \bK_i\rightarrow \bK_{m+i}$ is well defined by \eqref{e63} for all $i\geq 0$.
  The map $\tau _m(g)_i: \bK_i\rightarrow \bK_{m+i}$ is a {$1$-cocycle}
  if $i=0$ and a homomorphism if $i\ge 1$: indeed,
  for all $a,b\in K_i$, we have
  \begin{eqnarray*}
      \tau _m(g)\big((aK_{i+1})(bK_{i+1})\big)
      &=&\tau _m(g)(abK_{i+1}) \\
      &=&[g,ab]K_{m+i+1} \\
      &=& \big([g,a]\cdot {}^a[g,b]\big) K_{m+i+1}\\
      &=&
      \begin{cases}
	\tau _m(g)(a)+{}^{(aK_1)}\big(\tau _m(g)(b)\big)&\text{if $i=0$},\\
	\tau _m(g)(a)+\tau _m(g)(b)&\text{if $i\ge 1$}.
      \end{cases}
  \end{eqnarray*}

  Next, we verify that $(\tau _m(g)_i)_{i\geq 1}$ is  a derivation of    $\bar K_+$.
  For $a\in K_i$, $b\in K_j$, $i,j\ge 1$,  we have
  \begin{eqnarray*}
      \tau _m(g)\big([aK_{i+1},bK_{j+1}]\big)
      &=&\tau _m(g)\big([a,b]K_{i+j+1}\big)\\
      &=&[g,[a,b]]K_{m+i+j+1}\\
      &=&([ [g,a],b]K_{m+i+j+1})+([a,[g,b]]K_{m+i+j+1})\\
      &=&[ [g,a]K_{m+i+1},bK_{j+1}]+[aK_{i+1},[g,b]K_{m+j+1}]\\
      &=&[\tau _m(g)(aK_{i+1}),bK_{j+1}]+[aK_{i+1},\tau _m(g)(bK_{j+1})].
  \end{eqnarray*}

It remains to check that
  \begin{gather}
    \label{e64}
    \tau _m(g)\big({}^{(aK_1)}(bK_{i+1})\big)=\big[\tau _m(g)(aK_1),{}^{(aK_1)}(bK_{i+1})\big]+{}^{(aK_1)}\big(\tau _m(g)(bK_{i+1})\big)
  \end{gather}
  for $a\in K_0$ and $b\in K_i$, $i\ge 1$.  Indeed, {since}
  \begin{eqnarray*}
      {}^a[g,b]&=& [{}^ag,{}^ab]=[ [a,g]g,{}^ab] ={}^{[a,g]}[g,{}^ab]\cdot [ [a,g],{}^ab]\\
    &\equiv& [g,{}^ab]\cdot [ [g,a]^{-1},{}^ab]
    \equiv [g,{}^ab]\cdot [ [g,a],{}^ab]^{-1} \pmod{K_{m+i+1}},
  \end{eqnarray*}
  we obtain
  \begin{gather*}
    \begin{split}
    {}^{(aK_1)}\big(\tau _m(g)(bK_{i+1})\big)
    &={}^{(aK_1)}([g,b]K_{m+i+1})\\
    &=({}^a[g,b])K_{m+i+1}\\
    &=\big([g,{}^ab]\cdot [ [g,a],{}^ab]^{-1}\big) K_{m+i+1}\\
    &=\big([g,{}^ab]K_{m+i+1}\big)-\big([ [g,a],{}^ab]K_{m+i+1}\big)\\
    &=\tau _m(g)\big(({}^ab)K_{i+1}\big)-\big[ [g,a]K_{m+1},({}^ab)K_{i+1}\big]\\
    &=\tau _m(g)\big({}^{(aK_1)}(bK_{i+1})\big)-[\tau _m(g)(aK_1),{}^{(aK_1)}(bK_{i+1})],
    \end{split}
  \end{gather*}
  {proving} \eqref{e64}.    Thus, we have  $\tau _m(g)\in \Der_m(\bKb)$.

{Finally, we} show that the map $\tau _m: G_m\rightarrow \Der_m(\bKb)$ is a  homomorphism.
Indeed, for $g,g'\in G_m$, $a\in K_i$, $i\ge 0$, we have
  \begin{gather*}
    \begin{split}
      \tau _m(gg')(aK_{i+1})
      &=[gg',a]K_{m+i+1}\\
      &= \big({}^g[g',a]\cdot [g,a]\big) K_{m+i+1}\\
      &=\big([g',a]\cdot [g,a]\big) K_{m+i+1}\\
      &   =   [g',a] K_{m+i+1}  +  [g,a] K_{m+i+1}    \\
      &=\tau _m(g')(aK_{i+1})+\tau _m(g)(aK_{i+1}) =(\tau _m(g)+\tau _m(g'))(aK_{i+1}).
    \end{split}
  \end{gather*}
\end{proof}

It is easy to prove the following.

\begin{proposition}
  \label{r85}
  For $m\ge 0$, we have
  \begin{gather}
    \label{e65}
   \ker(\tau _m)=G_m\cap \modF ^{K_*}_{m+1}(G_0)
   =    \{g\in G_m\zzzvert [g,K_i]\subset K_{m+i+1} \text{ for $i\ge0$}    \},
  \end{gather}
   {where $\modF _*^{K_*}(G_0)$ is the Johnson filtration of $G_0$ induced by $K_*$.}
\end{proposition}

{Set ${\bar G}_m =  \gr_m(G_*)$ for {each} $m\geq 0$.}
{By Propositions \ref{r3} and \ref{r4},}
$\tau _m$ induces a  homomorphism
  \begin{gather}
    \label{e67}
    \bar\tau _m: {\bar G}_m \longrightarrow \Der_m(\bKb).
  \end{gather}
{By Proposition \ref{r85}, we have}
 \begin{gather}
    \label{e75}
    \ker(\bar\tau _m)=      (G_m\cap \modF _{m+1}^{K_*}(G_0)     ) /G_{m+1}.
  \end{gather}

\subsection{The Johnson morphism}   \label{sec:johnson-morphism}

In this subsection, we show that the family of all generalized Johnson
  homomorphisms {form}  a morphism of eg-Lie algebras,
which we call the \emph{Johnson morphism}.

\begin{theorem}
  \label{r87}
  {Let an extended N-series $G_*$ act}
  on an extended N-series $K_*$, and {set} ${\bar G}_\bu = \grb(G_*)$,  ${\bar K}_\bu = \grb(K_*)$.
  Then the family
  $\bar\tau _\bu=(\bar\tau _m)_{m\ge 0}$ of all homomorphisms $\bar\tau _m$ defined by \eqref{e67} is a morphism of eg-Lie algebras
  \begin{gather}
    \label{e4}
    \bar\tau _\bu: {\bar G}_\bu \longrightarrow \Der_\bu(\bKb).
  \end{gather}
  {Moreover}, $\bar\tau _\bu$ is injective if and only if
  {$G_*$   is the Johnson filtration $\modF _*^{K_*}(G_0)$.}
\end{theorem}

\begin{proof}
   We know  that $\bar\tau _m$ is a  homomorphism for {each} $m\geq 0$.
   Let us check that
  $(\bar\tau _m)_{m\ge 1}:{\bar G}_+ \to  \Der_+(\bKb)$
   preserves the Lie bracket. For
   $g\in G_m$, $g'\in G_n$, $m,n\ge 1$, $a\in K_i$, $i\ge 0$, we   have
  \begin{eqnarray*}
      &&\bar\tau _{m+n}([gK_{m+1},g'K_{n+1}])(aK_{i+1})\\
      &=&\bar\tau _{m+n}([g,g']K_{m+n+1})(aK_{i+1})\\
      &=&[ [g,g'],a]K_{m+n+i+1}\\
      &=& {\big[ [g,g'], [a,g']\cdot{}^{g'}\!a \big]  K_{m+n+i+1}}\\
      &=&   { \big[ [g,g'],{}^{g'}\!a \big]   K_{m+n+i+1} }\\
      &=&   {\big( \big[ {}^{g}\!{g'}, [a,g] \big] \cdot  \big[ {}^{a}\!g , [g',a] \big] \big) K_{m+n+i+1} }\\
      &=&   {\big( \big[ [g,g'] g' , [a,g] \big] \cdot  \big[ [a,g]g , [g',a] \big] \big) K_{m+n+i+1} } \\
      &=&   { \big( \big[ g' , [a,g] \big] \cdot  \big[ g , [g',a] \big]  \cdot  \big[ [a,g] , [g',a] \big]  \big)  K_{m+n+i+1} } \\
      &=&   { \big( \big[ g' , [g,a]^{-1} \big] \cdot  \big[ g , [g',a] \big]  \cdot  \big[ [g,a]^{-1} , [g',a] \big] \big) K_{m+n+i+1} } \\
      &=& -\bar\tau _n(g'G_{n+1})\big(\bar\tau _m(gG_{m+1})(aK_{i+1})\big) +       \bar\tau _m(gG_{m+1})\big(\bar\tau _n(g'G_{n+1})(aK_{i+1})\big)  \\
      && -\delta _{i,0}\big[\bar\tau _m(gG_{m+1})(aK_{i+1}),\bar\tau _n(g'G_{n+1})(aK_{i+1})\big]\\
      &=&\big[\bar\tau _m(gG_{m+1}),\bar\tau _n(g'G_{n+1})\big](aK_{i+1}).
  \end{eqnarray*}
  Hence $(\bar\tau _m)_{m\ge 1}$ is a morphism of graded Lie algebras.

  It remains to verify the equivariance property for $\bar\tau _\bu$.
  For $g\in G_0$, $g'\in G_m$, $m\ge 1$, $a\in K_i$, $i\ge 1$, we have
  \begin{gather*}
    \begin{split}
      \bar\tau _m\big({}^{(gG_1)}(g'G_{m+1})\big)(aK_{i+1})
      &=\bar\tau _m\big(({}^gg')G_{m+1}\big)(aK_{i+1})\\
      &=[{}^gg',a]K_{m+i+1}\\
      &={}^g\big[g',{}^{g^{-1}}\!a\big]K_{m+i+1}\\
      &=\bar\tau _0(gG_1) \Big(\bar\tau _m(g'G_{m+1})\big(\bar\tau _0(gG_1)^{-1}(aK_{i+1})\big)\Big)\\
      &=\big({}^{\bar\tau _0(gG_1)}\bar\tau _m(g'G_{m+1})\big)(aK_{i+1}).\\
    \end{split}
  \end{gather*}
  {Hence}   $\bar\tau _\bu$ is a morphism of eg-Lie algebras.

The second statement of the {theorem} says that $\bar\tau _m$ is
injective for all $m\ge 0$ if  and only if  we have
$G_m=\modF ^{K_*}_m(G_0)$ for all $m\ge 0$.
This equivalence is
easily checked by  induction on $m\geq 0$ using~\eqref{e75}.
\end{proof}

As a special case of Theorem \ref{r87}, we {obtain} the following.

\begin{corollary}
  \label{r6}
  Let $K_*$ be an extended N-series.  Then we have an injective
  morphism of eg-Lie algebras
\begin{equation}  \label{Johnson_Aut}
  \bar\tau _\bu: \grb(\Aut_*(K_*))\longrightarrow \Der_\bu(\grb(K_*)),
\end{equation}
{where $\Aut_*(K_*)$ is the Johnson filtration of $\Aut(K_*)$ defined by \eqref{A}.}
\end{corollary}

\begin{example}
  {Continuing} Examples \ref{r17} and \ref{r19}, 
    {let us} consider the adjoint actions $\Ad^{K_*}$ and $\ad^{\gr_\bu(K_*)}$.
   The morphism $\bar\tau_\bu$ in \eqref{Johnson_Aut} fits into the
  following commutative diagram:
  \begin{gather*}
    \xymatrix{
      \gr_\bu(K_*)\ar[rr]^{\gr_\bu(\Ad^{K_*})}\ar[rrd]_{\ad^{\gr_\bu(K_*)}}
      &&
      \gr_\bu(\Aut_*(K_*))\ar[d]^{\bar\tau_\bu}
      \\
      &&
      \Der_\bu(\gr_\bu(K_*)).
      }
  \end{gather*}
\end{example}

\section{Truncation of a derivation eg-Lie algebra} \label{sec:trunc-deriv-eg-1}

{Here we define the ``truncation'' $D_\bu(L_\bu)$ of the derivation
  eg-Lie algebra $\Der_\bu(L_\bu)$ of an eg-Lie algebra $L_\bu$.  This
  structure is useful {mainly} when the positive part $L_+$ of
  $L_\bu$ is a free Lie algebra generated by its {degree $1$} part.}

\subsection{Truncation of  a derivation eg-Lie algebra}  \label{sec:trunc-deriv-eg}

Let $L_\bu$ be an eg-Lie algebra.  Here we define a graded group
$D_\bu(L_\bu)=(D_m(L_\bu))_{m\ge 0}$, which we call the
\emph{truncation} of $\Der_\bu(L_\bu)$.  Set
\begin{gather}
  \label{e13}
\begin{split}
D_0(\Lbu)
&=  \{(d_0,d_1)\in \Aut(L_0)\times \Aut(L_1)\\
& \qquad \zzzvert d_1({}^ab)={}^{d_0(a)}(d_1(b))\hbox{ for $a\in L_0$, $b\in L_1$}\},
\end{split}
\end{gather}
which is a subgroup of $\Aut(L_0)\times \Aut(L_1)$.
For $m\ge 1$, {define an abelian group $D_m(\Lbu)$ by}
\begin{gather}
  \label{e15}
  \begin{split}
    D_m(\Lbu)&=      \{(d_0,d_1)\in Z^1(L_0,L_m)\times \Hom(L_1,L_{m+1})\\
    &\qquad  \zzzvert d_1({}^ab)=[d_0(a),{}^ab]+{}^a(d_1(b))\  \hbox{for } a\in L_0,b\in L_1      \},
  \end{split}
\end{gather}
where $Z^1(L_0,L_m)$ denotes the group of $L_m$-valued
  $1$-cocycles on $L_0$:
\begin{gather}
  Z^1(L_0,L_m)=    \{d_0: L_0\rightarrow L_m\zzzvert d_0(ab)=d_0(a)+{}^a(d_0(b))\text{ for $a,b\in L_0$}   \}.
\end{gather}

For every $m\geq 0$, there is a  homomorphism
\begin{gather}
  \label{e111}
  t_m: \Der_m(\Lbu)  \longrightarrow D_m(L_\bu),\quad (d_i)_{i\ge 0} \longmapsto (d_0,d_1).
\end{gather}

\begin{lemma}
  \label{r98}
  If the positive part $L_+$ of an eg-Lie algebra $L_\bu$ is
  generated by its degree $1$ part $L_1$, then  $t_m$ is injective for {each} $m\ge 0$.
\end{lemma}

\begin{proof}
  First, we prove that the kernel of $t_0$ is trivial. {Take}  $d=(d_i)_{i\ge 0}$ {such that}
  $(d_0,d_1)=(\id_{L_0},\id_{L_1})$.
  We {prove}   $d_i=\id_{L_i}$ for all $i\ge 0$ by  induction on~$i\geq 0$.  Let $i\ge 2$.
  Since $L_1$ generates $L_+$,  $L_i$ is generated by
  the elements $[x,y]$ with $x\in L_1$, $y\in L_{i-1}$.  We have
  \begin{gather*}
    d_i([x,y])=[d_1(x),d_{i-1}(y)]=[x,y]
  \end{gather*}
  by the induction hypothesis. Hence   $d_i=\id_{L_i}$.

   {Now we} prove that the kernel of $t_m$ is
  trivial {for $m\ge 1$}. {Take}   $d=(d_i)_{i\ge 0}$ {with}
  $(d_0,d_1)=(0,0)$.
  We {prove}   $d_i=0$ for all $i\ge 0$ by  induction on~$i\geq 0$.  Let $i\ge 2$.
  Since $L_1$ generates $L_+$,  $L_i$ is
  generated by the elements $[x,y]$ with $x\in L_1$, $y\in L_{i-1}$.  We
  have
  \begin{gather*}
    d_i([x,y])=[d_1(x),y]+[x,d_{i-1}(y)]=0
  \end{gather*}
  by the induction hypothesis. Hence   $d_i=0$.
\end{proof}

\begin{lemma}
  \label{r93}
Let $L_+=\bigoplus_{i\ge 1}L_i$
be the graded Lie algebra freely generated by an abelian group $A$ in degree $1$.
For $m \geq 1$, {every} homomorphism $d_1: A=L_1 \to L_{m+1}$ {extends}
(uniquely) to a derivation $d$ of $L_+$ of degree~$m$.
\end{lemma}

{This lemma is well known at least for $A$ a free abelian
  group. {(See \cite[Lemma~0.7]{Reutenauer} for instance.)}
  We give a proof here since we could not find a suitable reference  for the general case.}

\begin{proof}[Proof of Lemma \ref{r93}]
  Let $M=\bigoplus_{i\ge 1}M_i$ be the non-unital,
  non-associative algebra freely generated by  $A$ in degree $1$. (Thus we have  $M_1=A$, $M_2=A\otimes A$,
  $M_3=A\otimes (A\otimes A)\oplus(A\otimes A)\otimes A$, etc.)
  {Let $*: M\times M\rightarrow M$ denote}  the multiplication in $M$. Then the free Lie algebra $L_{+}$ may be defined as
  the quotient $M/I$ of $M$ by the ideal $I$ generated by the elements
  \begin{gather*}
    b*b,\quad b_1*(b_2*b_3)+b_2*(b_3*b_1)+b_3*(b_1*b_2)
  \end{gather*}
  for all $b,b_1,b_2,b_3\in M$.

\def\td{\ti{d}}
  Let $\td_1: M_1 \rightarrow M_{m+1}$ be a lift of $d_1$ to $M_{m+1}$, i.e.,
  we require that {the diagram}
  \begin{gather*}
    \xym{
      M_1\ar[d]_{\id_A}^{\cong}\ar[r]^{\td_1}
      &
      M_{m+1}\ar[d]^{p}
      \\
      L_1\ar[r]_{d_1}
      &
      L_{m+1}
    }
  \end{gather*}
  {commutes},  where $p$ denotes the   projection.  The map $\td_1$ extends uniquely to a
  degree $m$ derivation $\td_+=(\td_{i}: M_i\rightarrow M_{m+i})_{i\ge 1}$ of~$M$.
  {One easily checks}   $\td_+(I)\subset I$.
  Therefore, $\td_+$ induces a family of  homomorphisms
  $d_+=(d_i: L_i\rightarrow L_{m+i})_{i\ge 1}$.
  Clearly, $d_+$ is a degree $m$ derivation  of~$L_+$.
\end{proof}

\begin{proposition}
  \label{r92}
    If the positive part $L_+$    of an eg-Lie algebra $L_\bu$ is  freely generated by its degree  $1$ part $L_1$,
    then $t_m$ is an isomorphism for all $m\ge 0$.
\end{proposition}

\begin{proof}
  By {Lemma} \ref{r98}, $t_m$ is injective.  Thus it
  suffices to check that if $(d_0,d_1)\in D_m(L_\bu)$, then it extends to at
  least one $(d_i)_{i\ge 0}\in \Der_m(L_\bu)$.

{First, let $m=0$.}
The automorphism $d_1$ of  $L_1$ extends uniquely
to an automorphism $d_+=(d_i: L_i\rightarrow L_{i})_{i\ge 1}$ of the graded Lie algebra $L_+$.
It suffices to prove the equivariance property, i.e.,
  \begin{gather}
    \label{e87}
    d_i({}^ab)={}^{d_0(a)}(d_i(b))
  \end{gather}
 for $a\in L_0$, $b\in L_i$, $i\ge 1$, which is verified by  induction on $i\geq 1$.

{Now, let $m\geq1$.}
  By Lemma \ref{r93}, we can extend the  homomorphism $d_1$ to
   a derivation $d_+=(d_i: L_i\rightarrow L_{m+i})_{i\ge 1}$ of $L_+$ of degree $m$.
  It suffices to prove that
  \begin{gather}
    \label{e88}
    d_i({}^ab)=[d_0(a),{}^ab]+{}^a(d_i(b))
  \end{gather}
  for $a\in L_0$, $b\in L_i$, $i\ge 1$.
  The proof is by induction on $i\geq 1$.
  Let $i\ge 2$.  We may assume $b=[b',b'']$, $b'\in L_1$, $b''\in L_{i-1}$.  Then we have
  \begin{gather*}
    \begin{split}
    d_i({}^ab)
    &=d_i([{}^ab',{}^ab''])\\
    &=[d_1({}^ab'),{}^ab'']+[{}^ab',d_{i-1}({}^ab'')]\\
    &=\big[[d_0(a),{}^ab']+{}^a(d_1(b')),{}^ab''\big]+\big[{}^ab',[d_0(a),{}^ab'']+{}^a(d_{i-1}(b''))\big]\\
    &=\big[[d_0(a),{}^ab'],{}^ab''\big]+[{}^a(d_1(b')),{}^ab'']
    +\big[{}^ab',[d_0(a),{}^ab'']\big]    +[{}^ab',{}^a(d_{i-1}(b''))]\\
    &=\big[d_0(a),[{}^ab',{}^ab'']\big]+{}^a[d_1(b'),b'']+{}^a[b',d_{i-1}(b'')]\\
    &=\big[d_0(a),{}^a[b',b'']\big]+{}^a(d_i([b',b''])) =[d_0(a),{}^ab]+{}^a(d_i(b)),
    \end{split}
  \end{gather*}
  where the third identity is given by the induction hypothesis.
 \end{proof}

\subsection{The eg-Lie algebra structure of the truncation}  \label{sec:eg-lie-algebra}

Let $L_\bu$ be an eg-Lie algebra whose {positive part}
$L_+$ is freely generated by $L_1$.  By~Proposition \ref{r92},
$D_\bu(L_\bu)$ {is endowed} with a unique eg-Lie algebra structure such that
\begin{equation}  \label{t_map}
  t_\bullet =(t_m)_{m\ge 0}: \Der_\bu(L_\bu)\longrightarrow D_\bu(L_\bu)
\end{equation}
is an eg-Lie algebra isomorphism.
The following is easily derived from the definition of $\Der_\bu(L_\bu)$  given in Section \ref{sec:derivation-eg-lie}.

\begin{proposition}
  \label{r99}
  Let $L_\bu$ be an eg-Lie algebra such that $L_+$ is freely generated by $L_1$ as a graded Lie algebra.
  Then the graded group $D_\bu(L_\bu)$ has the
  following eg-Lie algebra structure.
\begin{enumerate}
  \item {The} Lie bracket $[d,d']\in D_{m+n}(L_\bu)$ of $d=(d_0,d_1)\in D_m(L_\bu)$ and $d'=(d'_0,d'_1)\in D_n(L_\bu)$
  with  $m,n\ge 1$ is defined by
    \begin{gather*}
      \begin{split}
      [d,d']_0(a)  &=  d_n(d'_0(a))-d'_m(d_0(a))-[d_0(a),d'_0(a)] \quad  \hbox{for } a\in L_0,\\
      {[d,d']_1(b)}  &= d_{n+1}(d'_1(b))-d'_{m+1}(d_1(b))\quad\quad
      \quad \quad \quad \quad   \hbox{for }  b\in L_1,
      \end{split}
    \end{gather*}
    where $d_+=(d_i)_{i\geq 1}$ and $d'_+=(d'_j)_{j\geq1}$
    {are the derivations of $L_+$ extending $d_1$ and $d'_1$, respectively.}
  \item {The} action  ${}^fd\in D_m(L_\bu)$ of $f=(f_0,f_1)\in D_0(L_\bu)$ on $d=(d_0,d_1)\in D_m(L_\bu)$
  with $m\ge 1$ is defined by
    \begin{gather*}
      \begin{split}
	({}^fd)_0(a)&=f_md_0f_0^{-1}(a)\quad \hbox{for } a\in L_0,\\
	({}^fd)_1(b)&=f_{m+1}d_1f_1^{-1}(b)\quad  \hbox{for } b\in L_1,
      \end{split}
    \end{gather*}
    where $f_+=(f_i)_{i\geq 1}$ is the automorphism of $L_+$ {extending}~$f_1$.
  \end{enumerate}
\end{proposition}

\section{Extended N-series associated with N-series}  \label{sec:en-series-associated-1}

{In this section, we} illustrate the constructions of the previous sections with the extended N-series defined by N-series.

\subsection{Extended N-series associated with N-series}   \label{sec:en-series-associated}

Let $K_+=(K_m)_{m\ge 1}$ be an N-series of a group $K=K_1$.
We consider here the extended N-series $K_*=(K_m)_{m\ge 0}$ obtained by setting $K_0=K_1=K$.

{By an}  \emph{action} of an extended N-series $G_*$ on
$K_+$ {we mean an action of $G_*$ on $K_*$.}

Let $L_+$ be a graded Lie algebra.  Let $\Der_0(L_+)
=\Aut(L_+)$ be the {automorphism group} of $L_+$
and, for $m\geq 1$, let $\Der_m(L_+)$ {denote} the
group of derivations of $L_+$ of degree $m$.
We call $\Der_+(L_+) =  (\Der_m(L_+)  )_{m\geq 1}$
the graded Lie algebra of \emph{positive-degree derivations} of $L_+$.
The group $\Aut(L_+)$ acts on $\Der_+(L_+)$ by {conjugation}.
Thus $\Der_\bu(L_+) =  (\Der_m(L_+)  )_{m\geq 0}$ is an eg-Lie algebra.

{Theorem \ref{r87} implies the following.}

\begin{corollary} \label{simpler}
{Let an extended N-series $G_*$ act}
  on an N-series $K_+$,
  and {let} ${\bar G}_\bu = \grb(G_*)$, ${\bar K}_+ =
  \gr_+(K_+)$.  Then the family $\bar\tau _\bu=(\bar\tau_m)_{m\ge 0}$ of all homomorphisms $\bar\tau _m$ defined
  by \eqref{e67} is a morphism of eg-Lie algebras
  \begin{equation} \label{Johnson_N}
    \bar\tau _\bu: {\bar G}_\bu \longrightarrow \Der_\bu(\bar K_+).
  \end{equation}
  {Moreover}, $\bar\tau _\bu$ is injective if and only if
 {$G_*$ is the Johnson filtration $\modF _*^{K_*}(G_0)$.}
\end{corollary}

In the rest of this section, we consider $N$-series with special
properties (called \emph{N$_0$-series} and \emph{N$_p$-series}).
We show   that if a group $G$ acts on such a
special N-series, then the positive part of the Johnson filtration of
$G$ is an N-series of the same kind.

\subsection{N$_0$-series}   \label{sec:0-restricted}

An \emph{N$_0$-series} of a group $K$ is an N-series $K_+$ such that
$K/K_{m}$ is torsion-free for all $m\geq 1$.

{An} N-series $K_+$ can be transformed {into} {an N$_0$}-series
${\sqrt{K_+}}$ {by considering the root sets of its successive terms.
Specifically, we define for all $m\geq 1$}
$$
{\sqrt{K_m}} =  \{ x\in K\, \vert\, x^i \in K_m \hbox{ for some } i\geq 1 \}.
$$ See \cite[\S IV.1.3]{Passi} {or \cite[\S 11, Lemma~1.8]{Passman} in
the case of the lower central series,} and \cite[Lemma
4.4]{Massuyeau_DSP} {in the general case}.  {Note that ${\sqrt{K_+}}$
is the smallest N$_0$-series of~$K$ containing $K_+$: thus,}
$\sqrt{K}_+ = K_+$ if and only if $K_+$ is {an N$_0$-series.}

\begin{example}
  \label{r18}
  The \emph{rational lower central series} of a group $K$ is the
  N$_0$-series $\sqrt{\Gamma _+K}=(\sqrt{\Gamma _mK})_{m\ge 1}$ associated to the
  lower central series $\Gamma _+K$ of $K$.  It is the smallest
  N$_0$-series of $K$.
\end{example}

\begin{proposition} \label{N_0}
Let a group $G$ act on an N$_0$-series~$K_+$.  Then {the positive
part $\modF _+^{K_*}(G)$ of the Johnson filtration
$\modF^{K_*}_*(G)$} is an N$_0$-series.
\end{proposition}

\begin{proof}
Set {$G_*=\modF _*^{K_*}(G)$.} {By}  Proposition~\ref{r77},
${G_+} $ is an N-series of ${G_1}$.
Therefore, it remains to show that ${G_m/G_{m+1}}$ is torsion-free for $m\geq 1$.
By Corollary \ref{simpler},  the {$m$th} Johnson homomorphism induces an injection
$$
\bar \tau_m: G_m/G_{m+1} \longrightarrow \Der_m(\bar{K}_+).
$$
Hence it suffices to check that  $\Der_m(\bar{K}_+)$ is torsion-free.
This follows {since} $\bar{K}_+$ itself is torsion-free.
\end{proof}

\subsection{N$_p$-series}  \label{sec:p-restricted}

Let $p$ be a prime.
An \emph{N$_p$-series} of a  group $K$ is an N-series  $K_+$ such that
$(K_m)^p \subset K_{mp}$ for all $m\geq 1$.
{By} a result of Lazard \cite[Corollary 6.8]{Lazard},  $\bar K_+ = \bigoplus_{i\geq 1} K_i/K_{i+1}$
is a restricted Lie algebra over the field $\mathbb{F}_p = \Z/p\Z$, whose $p$-operation
$$
(\cdot)^{[p]}: \bar K_+ \longrightarrow \bar K_+
$$
is defined by $(xK_{i+1})^{[p]}= (x^p K_{ip+1})$ for $x\in K_i$, $i\geq 1$.

Every N-series $K_+$ can be transformed into an N$_p$-series $K^{[p]}_+$  defined by
$$
K^{[p]}_m  =  \prod_{i\geq 1,\,  j\geq 0,\, ip^j \geq m} K_i^{p^j}  \quad \hbox{for } m\geq1.
$$
See \cite[\S IV.1.22]{Passi} or \cite[\S 11, Lemma 1.18]{Passman} {in the case of the lower central series}, and \cite[Lemma 4.6]{Massuyeau_DSP} {in the general case}.
{Note that  $K^{[p]}_+$ is the smallest N$_p$-series of $K$ containing $K_+$: thus,}
 $K^{[p]}_+=K_+$ if and only if $K_+$ is {an N$_p$-}series.

\begin{example}
  The \emph{Zassenhaus mod-$p$ lower central series} (also called   the \emph{Zassenhaus filtration}) of {a group} $K$
  \cite{Zassenhaus} is the N$_p$-series $\Gamma ^{[p]}_+K$
  associated to the lower central series $\Gamma _+K$ of $K$:
  \begin{equation} \label{Zassenhaus}
    \Gamma ^{[p]} _m K = \prod_{i\geq 1,\,  j\geq 0,\, ip^j \geq m}
    (\Gamma _i K)^{p^j}
    \quad\text{for $m\ge 1$}.
  \end{equation}
      {This ``mod-$p$'' variant of $\Gamma_+ K$ should not be confused
  with the \emph{Stallings mod-$p$ lower central series} (also called
  {the} \emph{lower exponent-$p$ central series}) $\Gamma ^{\la
  p\ra}_+K$ \cite{Stallings}, which is defined} inductively by $\Gamma
  ^{\la p\ra}_1 K = K$ and
\begin{gather}
\label{e78}\Gamma ^{\la p\ra}_{m+1}K=(\Gamma ^{\la p\ra}_mK)^{p}\,[K,\Gamma ^{\la p\ra}_mK]\quad\text{for $m\ge 1$}.
\end{gather}
{Indeed} {$\Gamma ^{[p]}_+K$ is the smallest N$_p$-series of~$K$,
whereas $\Gamma ^{\la p\ra}_+K$ is the smallest N-series $K_+$ of~$K$
such that $(K_m)^p\subset K_{m+1}$ for $m\ge 1$.}
\end{example}

\begin{proposition} \label{N_p}
Let a group $G$ act on an N$_p$-series~$K_+$.  Then {the positive
part $\modF _+^{K_*}(G)$ of the Johnson filtration $\modF^{K_*}_*(G)$} is an N$_p$-series.
\end{proposition}

\begin{proof}
Set {$G_*= \modF_*^{K_*}(G)$.}  {Since {$G_+$} is an N-series of
{$G_1$} by Proposition~\ref{r77}, it suffices} to {check} $ (G_m)^p
\subset G_{mp}$ for all $m\geq 1$.  Let $g\in G_m$ and $x\in K_j$,
$j\geq 1$. By Dark's commutator formula (see \cite[\S 11, Theorem
1.16]{Passman}), we have
$$
[g^p,x] = \prod_{i=1}^p c_i^{p \choose i },
$$
where $c_i$ is a product of iterated commutators, each {with}
at least $i$ components equal to $g^{\pm 1}$
and at least one component equal to $x^{\pm 1}$. It follows that
$$
c_i \in K_{j+im}, \quad \hbox{for $i=1,\dots,p$.}
$$
Therefore, $c_p \in  K_{j+pm}$ and, for $i\in \{1,\dots,p-1\}$, we have
$$
c_i^{p \choose i } \in  (K_{j+im})^p \subset K_{jp+imp} \subset K_{j+mp}
$$
since $p$ divides ${p \choose i}$ and $K_+$ is an N$_p$-series.
{Hence}  $[g^p,x]\in  K_{j+mp}$ and ${g^p \in G_{mp}}$.
\end{proof}

\begin{remark} \label{restricted}
{Let a group $G$ act} on an N$_p$-series $K_+$.
Since $\bar K_+$ is a restricted Lie algebra over $\mathbb{F}_p$,
so is $\Der_+(\bar K_+)$ with $p$-operation defined by  {the $p$-th power}.
Besides, {by} Proposition \ref{N_p},
$\gr_+ \modF_+^{K_*}(G)$ is a restricted Lie algebra over $\mathbb{F}_p$.
One can expect that the {positive part of the} Johnson morphism {in} Corollary \ref{simpler},
$$
\bar \tau_+: \gr_+ \modF_+^{K_*}(G) \longrightarrow \Der_+(\bar K_+),
$$ is a morphism of restricted Lie algebras (i.e., it preserves the
$p$-operations).  Furthermore, {it is plausible that} $\bar \tau_+$
takes values in the restricted Lie subalgebra of $\Der_+(\bar K_+)$
consisting of \emph{restricted} derivations {in the sense of Jacobson
\cite{Jacobson}.}

In degree $0$, it is easily verified that $\bar \tau_0:
\modF_0^{K_*}(G)/\modF_1^{K_*}(G) \to \Aut(\bar K_+)$ takes values in
the subgroup of automorphisms of the \emph{restricted} Lie algebra
$\bar K_+$.
\end{remark}

\begin{remark}
  \label{r11}
  An {\em exponent-$p$ N-series} of a group $K$ is an N-series $K_+$
  of $K$ such that $(K_m)^p\le K_{m+1}$ for all $m\ge 1$.  {For instance,} the Stallings
  mod-$p$ {lower} central series $\Gamma ^{\la p\ra}_+K$ of $K$
  {satisfies}  this property.
   We have {the following} variant of Proposition \ref{N_p}:
  \emph{If a group $G$ acts on an exponent-$p$ N-series $K_+$, then the positive part
  $\modF_+^{K_*}(G)$ of the Johnson filtration $\modF^{K_*}_*(G)$ is
  an exponent-$p$ N-series.}  The proof is easy and left to the reader.
\end{remark}

\section{The lower central series and its variants}
\label{sec:lower-central-series-1}

In this section, we {consider}
the lower central series $\Gamma _+K$ of a group $K$ and its
variants: the rational lower central series $\sqrt{\Gamma _+K}$, the Zassenhaus
mod-$p$ lower central series $\Gamma ^{[p]}_+K$ and the Stallings
mod-$p$ lower central series $\Gamma ^{\la p\ra}_+K$.

\subsection{The filtration $G^1_*$}\label{sec:filtration-g1}

Let $K_+$ be an N-series {of a group $K$}, and extend it to an extended N-series $K_*$
with $K_0=K$.  Let a group $G$ act on $K_*$, and let
$G_*=\modF _*^{K_*}(G)$ be the Johnson filtration of $G$ induced by $K_*$.
Define a descending series $G^1_*=(G^1_m)_{m\ge 0}$ of $G$ by
\begin{gather}
  \label{e9}
  G^1_m = \{g\in G\zzzvert [g,K]\subset K_{m+1}\}=\ker(G\rightarrow \Aut(K/K_{m+1})).
\end{gather}
Clearly, $G^1_m\ge G_m$ for $m\ge 0$, and $G^1_0=G=G_0$.

The filtration $G^1_*$ is not an extended N-series in general, {but it
is so for the lower central series and its variants.  In fact,
Andreadakis was the first to study the filtration $G^1_*$ in the case
of $K_+=\Gamma _+K$ with $G=\Aut(K)$, and he proved the following
proposition in this case \cite[Theorem 1.1.(i)]{Andreadakis}.}  
See also \cite[Lemma 3.7]{Cooper} {and \cite[proof of Theorem 2.4]{Paris}}
for $K_+=\Gamma ^{\langle p \rangle}_+K$,
and {see} \cite[Lemma 2.2.4]{MT} for $K_+=\Gamma ^{[p]}_+K$.

\begin{proposition}
  \label{r10}
  If $K_+$ is one of $\Gamma _+K$, $\sqrt{\Gamma _+K}$, $\Gamma
  ^{[p]}_+K$ and $\Gamma ^{\la p\ra}_+K$, then we have $G_*=G^1_*$.
  (In particular, $G^1_*$ is an extended N-series.)
\end{proposition}

\begin{proof}
To prove Proposition \ref{r10}, it suffices to check
$G^1_m\le G_m=\modF _{{m}}^{K_*}(G)$ for $m\ge 1$.  Thus, we need to check
\begin{gather}
  \label{e8}
  [G^1_m,K_n]\le K_{m+n} \quad \text{for $m\ge 1$, $n\ge 2$}.
\end{gather}
We prove \eqref{e8} in the four cases separately.

\begin{proof}[Case 1: $K_+=\Gamma _+K$.]
  Here we repeat Andreadakis' proof.
  We verify \eqref{e8} by induction on $n$ as follows:
  \begin{gather*}
    \begin{split}
      [G^1_m,K_n]
      &=[G^1_m,[K,K_{n-1}]]\\
      &\le \lala\;[[G^1_m,K],K_{n-1}]\cdot[[G^1_m,K_{n-1}],K]\;\rara_{K\rt G}\quad
      (\text{{by} Lemma \ref{3s}})\\
      &\le \lala\;[K_{m+1},K_{n-1}]\cdot[K_{m+n-1},K]\;\rara_{K\rt G}\quad(\text{{by the} induction hypothesis})\\
      &\le \lala\;K_{m+n}\;\rara_{K\rt G}=K_{m+n}.
    \end{split}
  \end{gather*}
\end{proof}

\begin{proof}[Case 2: $K_+=\sqrt{\Gamma _+K}$.]
  By induction on $n$, we will prove that $[g,a]\in K_{m+n}$ for
  $g\in G^1_m$ and $a\in K_n$.  We have $a^t\in \Gamma _nK$ for some $t\ge 1$.  We
  have
  \begin{gather*}
    \begin{split}
      &[g,a]^t
      \underset{\pmod{K_{m+n}}}\equiv
      \ \prod_{i=1}^t{}^{a^{i-1}}[g,a]
      =[g,a^t]\in [G^1_m,\Gamma _nK]\le [G^1_m,[K,K_{n-1}]],
    \end{split}
  \end{gather*}
  where $\equiv$ follows from
  \begin{gather*}
    [a^{i-1},[g,a]]\in [K_n,[G^1_m,K]]\le [K_n,K_{m+1}]\le K_{m+n+1}\le K_{m+n}.
  \end{gather*}
  Similarly to Case 1, we obtain $[G^1_m,[K,K_{n-1}]]\le K_{m+n}$ using
  the induction hypothesis.  Therefore, we have $[g,a]^t\in K_{m+n}$,
  hence $[g,a]\in K_{m+n}$.
\end{proof}

\begin{proof}[Case 3:  $K_+=\Gamma _+^{[p]}K$.]
By \eqref{Zassenhaus}, it suffices to prove by induction on $n$ that
$[g,z^{p^j}]\in K_{m+n}$ if $g\in G^1_m$, $z\in \Gamma _iK$, $i\ge 1$, $j\ge 0$ and
$ip^j\ge n$.

If $j=0$, then $z\in \Gamma _iK\le \Gamma _nK = [K,\Gamma _{n-1}K]\le [K,K_{n-1}]$.  Then we
proceed as in Case 1 using the induction hypothesis.

Let $j\ge 1$.  By Dark's commutator formula (see \cite[\S 11,
Theorem~1.16]{Passman}), we~have
\begin{equation} \label{prod_bin}
 [g,z^{{p^j}} ] = \prod_{d=1}^{p^j} c_d^{p^j \choose d },
\end{equation}
where $c_d$ is a product of iterated commutators, each with at least
$d$ components equal to $z^{\pm 1}$ and at least one component equal
to $g^{\pm 1}$.  We can assume without loss of generality that $i$ is the least integer greater than or equal to $n/p^{j}$,
{so that~$i<n$.}  By $z\in \Gamma _i K\le  K_i$ and the induction hypothesis, we
have $[g^{\pm 1}, z^{\pm 1}] \in K_{m+i}$.  It follows that
\begin{equation} \label{m+dr}
c_d \in K_{m+di}.
\end{equation}

For each $k\geq 1$, let $\vert k\vert_p$ denote the \emph{$p$-part} of
$k$, which is the unique power of $p$ such that $k/|k|_p$ is an
integer coprime to $p$.  Then we have $ \left\vert {p^j \choose d}
\right\vert_p \geq \frac{p^j}{\vert d\vert_p} $ (see, e.g., the proof
of \cite[\S 11, Lemma 1.18]{Passman}).  Therefore,
$$
\left\vert{p^j \choose d}\right\vert_p(m+di) \geq   \frac{p^j}{d} (m+di) \geq \frac{p^j}{d} m + p^j i\geq m+n.
$$
Since $K_+$ is an N$_p$-series, \eqref{m+dr} implies $c_d^{p^j \choose
d } \in K_{m+n}$.  Hence, by \eqref{prod_bin}, we have
$[g,z^{{p^j}}]\in K_{m+n}$.
\end{proof}

\begin{proof}[Case 4:  $K_+=\Gamma _+^{\la p\ra}K$.]
By \eqref{e78}, it suffices to prove by induction on $n$ that we have
$[G^1_m,[K,K_{n-1}]]\subset K_{m+n}$ and $[G^1_m,(K_{n-1})^p]\subset
K_{m+n}$.  The former is proved
similarly to Case 1 by using the
induction hypothesis; to prove the latter, we will verify
$[g,z^p]\in K_{m+n}$ for $g\in G^1_m$ and $z\in K_{n-1}$.  We have
\begin{gather*}
  \begin{split}
    [g,z^p]
    =\prod_{i=1}^p{}^{z^{i-1}}[g,z]
    \underset{\pmod{K_{m+n}}}\equiv
    [g,z]^p
    \in [G^1_m,K_{n-1}]^p\le (K_{m+n-1})^p\le K_{m+n},
  \end{split}
\end{gather*}
where $\equiv$ follows from
\begin{gather*}
  [z^{i-1},[g,z]]\in [K_{n-1},[G^1_m,K_{n-1}]]\le [K_{n-1},K_{m+n-1}]\le K_{m+2n-2}\le K_{m+n}.
\end{gather*}
Hence $[g,z^p]\in K_{m+n}$.
\end{proof}

This completes the proof of Proposition \ref{r10}.
\end{proof}

{We now observe that the Johnson filtration $G^1_*=G_*$ can be given a
ring-theoretic description, in the case of the rational (resp.
Zassenhaus {mod-$p$}) lower central series.}

\begin{corollary}
  \label{r091}
  If $K_+=\sqrt{\Gamma _+K}$, then for $m\ge 0$ we have
  \begin{eqnarray*}
    G_m=G^1_m
    = \ker(\Aut(K)\rightarrow \Aut(\Q[K]/ I^{m+1})),
  \end{eqnarray*}
where $I = \ker (\epsilon:\Q[K] \to \Q)$ is the augmentation ideal.
\end{corollary}

\begin{proof}
{This follows from a classical result of Malcev, Jennings and P. Hall, which computes the ``dimension subgroups'' with coefficients in $\Q$:}
$$
(1+ I^{m+1}) \cap K = \sqrt{\Gamma _{m+1} K}   \ \subset \Q[K] \quad \hbox{for $m\geq 0$}.
$$
(See, e.g., \cite[\S IV.1.5]{Passi} or \cite[\S 11, Theorem 1.10]{Passman}.)
\end{proof}

\begin{corollary}
  \label{r191}
  If $K_+=\Gamma ^{[p]}_+K$, then for $m\ge 0$ we have
  \begin{eqnarray*}
    G_m=G^1_m= \ker (G\rightarrow \Aut(\mathbb{F}_p[K]/I^{m+1}) ),
  \end{eqnarray*}
where $I = \ker (\epsilon:\mathbb{F}_p[K] \to \mathbb{F}_p)$ is the
augmentation ideal.
\end{corollary}

\begin{proof}
{This follows from a classical result of Jennings and Lazard, which
computes the ``dimension subgroups'' with coefficients in
$\mathbb{F}_p$:}
$$
(1+ I^{m+1}) \cap K  = \Gamma _{m+1}^{[p]} K  \ \subset \mathbb{F}_p[K] \quad \hbox{for $m\geq 0$}.
$$ (See, e.g., \cite[\S IV.2.8]{Passi} or \cite[\S 11, Theorem
1.20]{Passman}.)
 \end{proof}

\subsection{Examples and remarks}

{In the light of Proposition \ref{r10},
we now relate the  results and constructions of the previous sections
to those in the  literature.}

\begin{example} \label{Andreadakis}
Andreadakis \cite{Andreadakis} mainly considered the case
{where $K_+ = \Gamma_+ K$ is the lower central series of a free group
  $K$ and $G=\Aut(K)$. ({By}   Lemma~\ref{r96},   $G$ acts on $K_+$.)}
{In this case, the Johnson filtration $\Aut_*(K_*)=G_*=G_*^1$ is usually called the \emph{Andreadakis--Johnson filtration}.}
{Note that} $\bar K_+$ is
the free Lie algebra $\Lie(K^{\abel})$ on the abelianization
$K^{\abel}=K/\Gamma _2K$.  Hence, by Proposition \ref{r92}, the eg-Lie
algebra morphism~\eqref{t_map}
$$
t_\bu :  \Der_\bu ( \bar K_+) \longrightarrow D_\bu(\bar K_+)
$$ is an isomorphism, where $D_\bu(\bar K_+) = (D_m(\bar K_+) )_{m\geq
0}$ is given by
$$
D_0(\bar K_+) = \Aut(K^{\abel}) \quad \hbox{and} \quad  D_m(\bar K_+) = \Hom(K^{\abel}, \Lie_{m+1} (K^{\abel}))
$$
and has the eg-Lie algebra structure described in Proposition \ref{r99}.
For a finitely generated free group $K$, the composition 
$$
{t_\bu \bar\tau _\bu: \gr_\bu (G_*) \longrightarrow D_\bu(\bar K_+)}
$$
has been extensively studied since Andreadakis' work;
we refer to \cite{Satoh_free} for a survey.
\end{example}

\begin{example} \label{Johnson}
Let $\Sigma_{g,1}$ be a compact, connected, oriented surface of genus
$g$ with one boundary component, and let $K=\pi_1(\Sigma_{g,1},\star)$, where $\star\in \partial\Sigma_{g,1}$.  
The mapping class group
$$
G = \operatorname{MCG}(\Sigma_{g,1}  , \partial  \Sigma_{g,1}   )   
$$ of $\Sigma_{g,1}$ relative to $\partial\Sigma_{g,1}$ 
{acts on $K_+=\Gamma_+ K$ in the natural way}.  
By Proposition~\ref{r10}, the Johnson filtration $G_*$ 
{in our sense} coincides with the \emph{Johnson filtration}~{$G^1_*$} in the usual sense,
{and its first term}  $G_1 =G^1_1 =\ker(G\rightarrow \Aut(H_1(\Sigma _{g,1};\modZ )))$ is
known as the \emph{Torelli group.}  By Example \ref{Andreadakis}, we
have an injective morphism of eg-Lie algebras
$t_\bu \bar\tau _\bu: \gr_\bu (G_*) \to D_\bu(\bar K_+)$.
The components
$$
t_m \tau _m: G_m \longrightarrow \Hom(H, \Lie_{m+1} (H))
$$ for $m\ge 1$, where $H=H_1(\Sigma_{g,1};\Z)$, are the original
\emph{Johnson homomorphisms} introduced by Johnson
\cite{Johnson_tau1,Johnson_survey} and Morita \cite{Morita}.  See
\cite{Satoh_surface} for a survey.
\end{example}

\begin{remark}
  \label{r12}
  {(i)} Since the rational lower central series of a free group
  coincides with {the} lower central series, we could
  replace the latter by the former in Examples~\ref{Andreadakis}
  and~\ref{Johnson}. Thus, Corollary~\ref{r091} implies that the
  Johnson filtration of the mapping class group of~$\Sigma_{g,1}$
  (resp. the Andreadakis--Johnson filtration of the automorphism group
  of a free group) can be described using Fox's free differential
  calculus \cite{Morita,Perron_0}.

  {(ii)} Example \ref{Johnson} can be adapted to a closed oriented
  surface $\Sigma_g$ of genus~$g$.  In this case, additional
  technicalities arise since $K=\pi_1(\Sigma_g)$ is not free, and the
  mapping classes of $\Sigma_g$ {act on $K$ as outer} automorphisms.
  The Johnson homomorphisms in {this} case were introduced by Morita
  \cite{Morita_closed}.
\end{remark}

\begin{example} \label{mod-p_Johnson}
As in Example \ref{Johnson}, {we consider} the mapping class group $G=
\hbox{MCG}(\Sigma_{g,1},\partial \Sigma_{g,1})$ {acting} on
$K=\pi_1(\Sigma_{g,1},\star)$.  Here, let $K_+$ be one of the two
versions of the mod-$p$ lower central series.  Note that the
associated graded Lie algebra $\bar K_+$ is defined over
$\mathbb{F}_p$ {in both cases}. If $K_+=\Gamma ^{[p]}_+K$
(resp$.$ $\Gamma ^{\la p\ra}_+K$), then the Johnson filtration $G_*$
induced by $K_*$ coincides with
{the ``Zassenhaus (resp$.$ Stallings) mod-$p$ Johnson filtration'' considered by Cooper in~\cite{Cooper}}, {and its
first term} $G_1=\ker(G\rightarrow \Aut(H_1(\Sigma_{g,1};\mathbb{F}_p)))$ is the \emph{mod-$p$ Torelli group}.
Furthermore, the  {``{$m$th} Zassenhaus (resp$.$ Stallings) mod-$p$
Johnson homomorphism'' for $m\geq 1$ defined in \cite{Cooper}} coincides with the composition
$$
G_m \stackrel{\tau _m}{\longrightarrow} \Der_m(\bar K_+)  \stackrel{t_m}{\longrightarrow}
\Hom(\bar K_1, \bar K_{m+1}).
$$
{According to Proposition \ref{r10}, we have $\ker (t_m \tau_m)=G_{m+1}$.}
{In fact, these constructions for $K_+=\Gamma ^{\langle p\rangle}_+  K$ 
{had been} used by Paris {\cite{Paris}} to prove that the mod-$p$ Torelli group (of an arbitrary compact, oriented surface) is residually a $p$-group.}

Now let us focus on the case $K_+=\Gamma ^{[p]}_+ K$.  In this case, the
graded Lie algebras~$ \bar K_+$, $\Der_+ ( \bar K_+)$ and $\bar G_+$
are restricted over $\mathbb{F}_p$.  (See~Proposition
\ref{N_p} and Remark \ref{restricted}.)  Since $K$ is a free group, 
$\bar K_+$ is the free restricted Lie algebra over~$\mathbb{F}_p$ generated by 
$\bar K_1 \simeq H_1(\Sigma_{g,1}; \mathbb{F}_p)$  \cite[Theorem 6.5]{Lazard}.
{By} Corollary \ref{r191}, we can describe $G_*$ using Fox's free
differential calculus, so that $G_*$ coincides with Perron's ``modulo
$p$ Johnson filtration'' \cite{Perron_p}.  (See {also} \cite[Theorem 4.7]{Cooper} in this connection.)  
{By Proposition \ref{N_p},}
this filtration satisfies $(G_m)^p \subset G_{mp}$ for $m\geq 1$; this fact does not seem to have been
observed before.
\end{example}

{
\begin{remark}
  \label{r16}
  It seems plausible that one can adapt the constructions of this
  paper to the setting of (extended) N-series {of} 
  \emph{profinite} groups and, in particular, \emph{pro-$p$} groups.
  {In} fact, the literature offers
  several such constructions for the lower central series of a pro
  $p$-group, or its variants.  For instance, Asada and Kaneko~\cite{AK} introduced analogues of the Johnson
  homomorphisms on the automorphism group of the pro-$p$ completion of
  a surface group. More recently, Morishita and Terashima~\cite{MT}
  studied the Johnson homomorphisms for the
  automorphism group of the Zassenhaus filtration of a finitely
  generated pro-$p$ groups, which may be regarded as variants of
  Cooper's ``Zassenhaus mod-$p$ Johnson homomorphisms''.
\end{remark}
}

\section{Two types of series associated with pairs of groups}  \label{sec:en-series-associated-2}

In this section, we consider two types of series $K_*=(K_m)_{m\geq 0}$
determined by {their} first few terms: the smallest 
extended N-series {with given}  $K_0$ and~$K_1$, and the {smallest}  N-series {with given} $K_0=K_1$ and $K_2$.

\subsection{Extended N-series determined by $K_0$ and $K_1$}\label{sec:first-way}

Let $K=K_0$ be a group, and {let $K_1\trll K$}.
Define an extended N-series $K_*=(K_m)_{m\ge 0}$ by
\begin{gather}
  \label{e106}
  K_m =
  \begin{cases}
    K &\text{if $m=0$},\\
    \Gamma _m K_1&\text{if $m\ge 1$}.
  \end{cases}
\end{gather}
Note that $K_*$ is the smallest extended N-series with these $K_0$ and $K_1$.
{The} eg-Lie algebra  $\bKb=\grb(K_*)$ associated to $K_*$ is given by
$$
\bK_0= K_0 / K_1  \quad \hbox{and} \quad \bK_m= \Gamma _m  K_1 /\Gamma _{m+1}  K_1 \quad \hbox{for } m\geq 1.
$$

Let a group $G$ act on $K$  in such a way that ${}^GK_1=K_1$.
Since $K_i=\Gamma _i K_1$ is {characteristic in} $K_1$ for all $i\geq 1$,
$G$ acts on the extended N-series $K_*$ {(see Lemma~\ref{r96})}.
{Define three descending series $G^0_*$, $G^1_*$ and $G_*$  of $G$ by
\begin{align*}
G^0_m &= \{g\in G\zzzvert [g,K_0]\subset K_m\}=\ker(G\rightarrow \Aut(K_0/K_m)),\\
G^1_m &= \{g\in G\zzzvert [g,K_1]\subset K_{m+1}\}=\ker(G\rightarrow \Aut(K_1/K_{m+1})),\\
G_m &= G^0_m\cap G^1_m = \{g\in G\zzzvert [g,K_0]\subset K_m,\;[g,K_1]\subset K_{m+1}\}.
\end{align*}
The Johnson filtration $\modF ^{K_*}_*(G)$ has the following simpler
description.

\begin{proposition}
  \label{r105}
  We have $G_*=\modF ^{K_*}_*(G)$.  (Hence $G_*$ is an extended N-series.)
  Moreover, $G^1_*$ is an extended N-series.
\end{proposition}

\begin{proof}
By Proposition \ref{r10}, we have
\begin{gather}
  \label{e10}
  G^1_m = \{g\in G\zzzvert [g,K_n]\subset K_{m+n}\text{ for $n\ge 1$}\}
\end{gather}
for $m\ge 0$, and $G^1_*$ is an extended N-series.
By \eqref{e10}, we have
\begin{gather}
  G_m = G^0_m\cap G^1_m = \modF _m^{K_*}(G)
\end{gather}
for $m\ge 0$.
\end{proof}

Since $G^1_m\ge G^0_{m+1}$ for $m\ge 0$, the filtrations $G_*$ and $G^0_*$
are nested:
\begin{equation} \label{nested}
G=G^0_0=G_0\ge G^0_1\ge G_1\ge \cdots \ge G_{m-1} \ge G^0_m\ge G_m\ge \cdots.
\end{equation}
}

\begin{theorem} \label{G[m]=G<m>}
{If $K_1$ is} a {non-abelian} free group, then,
for {each}~$m\geq 0$, we have
\begin{gather*}
  G_m=G^1_m\le G^0_m.
\end{gather*}
\end{theorem}

\begin{proof}
 Note that $G^1_m\le G^0_m$ implies $G_m=G^1_m$.  Hence it suffices
  to prove by induction on $m \geq 0$ that if $g\in G$,
  $[g,{K_1}]\subset K_{m+1}$, then $[g,{K_0}] \subset K_m$.  The case
  $m=0$ is trivial; let $m\ge 1$.  Let $y\in {K_0}$.  By the induction
  hypothesis, we have $[g,y]\in [g,K_0]\subset K_{m-1}$, i.e.,
  $g(y)=zy$ for some $z\in K_{m-1}$.  For each $x\in K_1$, we have
  \begin{gather*}
    {}^y x
    \equiv g({}^y x)
    = {}^{g(y)} g(x)
   \equiv {}^{g(y)} x
    ={}^{zy}x
    =[z,{}^yx]\;{}^yx \pmod{K_{m+1}},
  \end{gather*}
  where each $\equiv$ {follows}   from $[g,K_1]\subset K_{m+1}$.
  {Therefore} $[z,K_1]\subset K_{m+1}$.
   By Lemma~\ref{r043} below, we  {have} $z\in K_m$ and {hence} $[g,y] \in K_m$.
\end{proof}

\begin{lemma}
  \label{r043}
  If  $F$ is a {non-abelian} free group {and $m\ge 1$,} then we have
  \begin{gather*}
    \{a\in F\zzzvert [a,F]\subset \Gamma _{m+1} F\} = \Gamma _m F.
  \end{gather*}
\end{lemma}

\begin{proof}
  Let $L_m  =  \{a\in F\zzzvert [a,F]\subset \Gamma _{m+1} F\}$. We
  will prove  $L_m = \Gamma _m F$ for $m\geq 1$ by
  induction. Let $m\geq 2$.  Clearly, $\Gamma _m F \le   L_m$. By
  the induction hypothesis, we have $L_m \le  L_{m-1}  \le   \Gamma _{m-1}F$.  
  The quotient group $L_m/\Gamma _m F$, regarded as a  subgroup~of
  $$
  \Gamma _{m-1} F/\Gamma _m F  \simeq \Lie_{m-1}(F^\abel), \quad \hbox{where } F^\abel  =  F/\Gamma _2 F,
  $$
  is the centralizer of $\Lie_1(F^\abel)=F^\abel$ in the free Lie
  algebra $\Lie(F^\abel)$.  Since $\hbox{rank}(F^\abel) \geq 2$,  the
  center of $\Lie(F^\abel)$ is trivial. Hence  $L_m/\Gamma _m F$ is trivial.
\end{proof}

\begin{remark}
  \label{r9}
  Lemma \ref{r043} can be restated as follows.  Let $F$ be a
  non-abelian free group, let $F_+=\Gamma _+F$ be its lower central series,
  and extend $F_+$ to an extended N-series $F_*$ with $F_0=F_1$.
  {Then, letting $F$ act on $F_*$ by conjugation,  the Johnson filtration of $F$ induced by $F_*$ coincides with $F_*$.}
\end{remark}

In what follows, \emph{{let} $K_1$ {be} a {non-abelian} free group}.
Then $\bar K_+ = (K_m/K_{m+1})_{m\geq 1}$ is the free Lie
algebra on $\bar K_1 = K_1^\abel$.  By {Theorem \ref{r87} and Proposition \ref{r99},}
we {obtain} an injective eg-Lie algebra morphism
 \begin{gather*}
   {\bar G}_\bu \stackrel{\bar\tau _\bu} {\longrightarrow} \Der_\bu(\bKb)  \mathop{\longrightarrow}^{t_\bu}_\simeq  D_\bu(\bKb),
  \end{gather*}
where ${\bar G}_\bu  =    ({G_m}/{G_{m+1}}  )_{ m\geq 0}$. 
By \eqref{e13} and \eqref{e15}, the {$m$th}  Johnson homomorphism
$t_m\tau_m: {G_m} \to D_m(\bar K_\bu)$ has two components
$$
\tau _0^{{0}} : {G_0}  \longrightarrow  \Aut(K_0/K_1), \quad \tau _0^{{1}}: {G_0}  \longrightarrow  \Aut({K_1^\abel})
$$
for $m=0$, and
$$
\tau _m^{{0}} : {G_m}  \longrightarrow  Z^1 (K_0/K_1, \Lie_m({K_1^\abel})  ),
\quad  \tau _m^{{1}} : {G_m}  \longrightarrow  \Hom ({K_1^\abel}, \Lie_{m+1}({K_1^\abel}) )
$$
for $m\geq 1$. {Furthermore, these two components}  
are related to each other by 
\begin{gather}
  \label{e16}
  \tau^1_m(g)({}^ab)=
  \begin{cases}
    {}^{\tau^0_0(g)(a)}\big(\tau^1_0(g)(b)\big)& (m=0),\\
    \big[\tau^0_m(g)(a),{}^ab\big]+{}^a\big(\tau^1_m(g)(b)\big)&(m\ge1)
  \end{cases}
\end{gather}
for {$g\in G_m$,} $a\in K$, $b\in K_1$.  {Note also that}
$$
\ker \tau _m^{0} =  {G^0_{m+1}}, \quad  \ker \tau _m^{1} = {G^1_{m+1}}  = {G_{m+1}}
\qquad {(m\geq 0)}.
$$

\begin{proposition}
  \label{r20}
The homomorphism
$\tau^1_0$ restricts to
\begin{gather*}
  \tau^1_0|_{G^0_1}:G^0_1\lto\Aut_{\Z[K/K_1]}(K_1^\abel).
\end{gather*}
For $m\ge1$, the homomorphism
$\tau^1_m$ restricts to
\begin{gather*}
  \tau^1_m|_{G^0_{m+1}}:G^0_{m+1}\lto\Hom_{\Z[K/K_1]}(K_1^\abel,\Lie_{m+1}(K_1^\abel)).
\end{gather*}
\end{proposition}

\begin{proof}
{This immediately follows from \eqref{e16}.}
\end{proof}

\begin{proposition}
  \label{r5}
Let $m\geq 1$. There is a map
$$
\ti\tau_m^0: G^0_{m} \lto Z^1 (K_0, \Lie_m({K_1^\abel})  )
$$
which is a homomorphism for $m\geq 2$ (resp., a $1$-cocycle for $m=1$) with kernel~$G^0_{m+1}$,
and which makes the following diagram commute:
\begin{equation} \label{comd}
\xymatrix{
G^0_{m} \ar@{-->}[rr]^-{\ti\tau_m^0}  &&  Z^1 (K_0, \Lie_m({K_1^\abel})  )  \\
{G_{m}} \ar[rr]^-{\tau_m^0}  \ar@{->}[u]&&  Z^1 (K_0/K_1, \Lie_m({K_1^\abel})  ) \ar[u]
}
\end{equation}
(Here the {arrow on the left} is the inclusion, and {that on the right} is induced by the projection $K_0 \to K_0/K_1$.)
\end{proposition}

\begin{proof}
For  $g\in G^0_m$, {the map}  $g': K_0 \to K_m/K_{m+1} \simeq \Lie_m({K_1^\abel})$
defined by $g'(x) =  [g,x]$ is a $1$-cocycle.
Thus the map $\ti\tau_m^0: G^0_{m} \to Z^1 (K_0, \Lie_m({K_1^\abel})  )$ defined by ${\ti\tau}_m^0(g)=g'$
{makes}  {the diagram \eqref{comd} commute}. For  $g,h \in G_m^0$ and $x\in K_0$, we have
$$
(gh)'(x) = {[gh,x]}K_{m+1} =   {{}^{g}[h,x][g,x]}K_{m+1} = {}^{g}(h'(x)) + g'(x).
$$
{Hence,} ${\ti\tau}_m^0$ is a $1$-cocycle for $m=1$, and a homomorphism for $m>1$. Clearly, its kernel is $G^0_{m+1}$.
\end{proof}

{We now  illustrate the above constructions with a few examples.} 

\begin{example}
  \label{r14}
  {As in Example \ref{Johnson}, {we consider} the mapping class group 
  $G= \hbox{MCG}(\Sigma_{g,1},\partial \Sigma_{g,1})$ {acting} on
$K=\pi_1(\Sigma_{g,1},\star)$.   
If $H:=K_1$ is a characteristic subgroup of $K$, then
 the filtration $(G^0_m)_{m\ge1}$ of $G^0_1$ coincides with the ``higher
  order Johnson filtration'' defined by McNeill \cite{McNeill},
and  her ``higher order Johnson homomorphism'' $\tau_m^H$
 coincides with our $\tau^1_{m-1}|_{G^0_m}$ {for}  $m\geq 2$.
When $H=\Gamma_2 K$, the subgroup $G^0_1$ of $G$ is the Torelli group,
and $G^0_2$ is  the kernel of the so-called ``Magnus representation'': 
the study of this case is carried out in \cite{McNeill}.}
\end{example}

\begin{example} \label{fake_handlebody}
Let $K_0=\langle x_1,\ldots ,x_p,y_1,\ldots ,y_q\rangle $ be the free
group of rank $p+q$, $p,q\ge0$.
Set $K_1=\lala x_1,\ldots ,x_p\rara\trll K_0$.  We have $K_0/K_1\simeq
F_q:=\langle y_1,\ldots ,y_q\rangle $.  Let~$K_*$ be the extended N-series defined by
\eqref{e106}. We call
\begin{gather*}
  G=\Aut(K_*)=\{f\in \Aut(K_0)\;|\;f(K_1)=K_1\}
\end{gather*}
the \emph{fake handlebody group} of type $(p,q)$, and 
$$
G^0_1 = \ker  ( G \longrightarrow \Aut(F_q)   )
$$
the \emph{fake twist group} of type $(p,q)$; see Example
\ref{handlebody} below to {clarify}
this terminology.
If $p\ge 1$ and $(p,q)\neq(1,0)$, then $K_1$ is a non-abelian free
group, and {Theorem~\ref{G[m]=G<m>} applies.}
We will study these groups in more details in \cite{HM_tg2}
using the Johnson homomorphisms $(\tau_m^1)_{m\geq 0}$  and
  $({\ti\tau}_m^0)_{m\geq 0}$ defined on
the two nested filtrations $({G_m})_{m\geq 0}$ and $({G^0_m})_{m\geq0}$, respectively.
\end{example}

\begin{example} \label{handlebody}
Let $V_g$ be a handlebody of genus $g\geq 1$, fix a disk $S\subset \partial V_g$
and let $\Sigma_{g,1}  = \partial V_g \setminus \operatorname{int}(S)$.
{Let}  $\star \in \partial \Sigma_{g,1}$ and set
$$
K_0  = \pi_1(\Sigma_{g,1},\star) \quad \hbox{and} \quad K_1  = \ker ( i_*: \pi_1(\Sigma_{g,1},\star) \longrightarrow \pi_1(V_g,\star) ),
$$
where $i_*$ is induced by the  inclusion $i:\Sigma_{g,1} \hookrightarrow V_g$. Let   $\operatorname{MCG}(\Sigma_{g,1}, \partial \Sigma_{g,1})$ act
  on~$K_0$ in the canonical way. The subgroup
$$
G =  \{ f\in \operatorname{MCG}(\Sigma_{g,1}, \partial \Sigma_{g,1}) \  \vert\ f_*(K_1)=K_1  \}
$$
is usually called the \emph{handlebody group},
{since it is}
the image of $\operatorname{MCG}(V_{g}, S)$
in $\operatorname{MCG}(\Sigma_{g,1}, \partial \Sigma_{g,1})$
by the restriction homomorphism (which is injective). 
The subgroup
\begin{gather*}
  \begin{split}
{G^0_1} &= \ker  ( G  \longrightarrow \Aut(K_0/ K_1 )   )\\
&\simeq \ker  ( \operatorname{MCG}(V_{g}, S) \longrightarrow \Aut( \pi_1(V_g,\star))  ),
  \end{split}
\end{gather*}
usually called the \emph{twist group},
is generated by  Dehn twists along  the boundaries of 2-disks properly
embedded in~$V_g\setminus S$ \cite{Luft}.
The present example corresponds to Example \ref{fake_handlebody}
{with} $p=q=g$, {where} the basis $(x_1,\dots,x_g,y_1,\dots,y_g)$  of $K_0$ is a system of meridians and parallels on $\Sigma_{g,1}$,
and the automorphisms of $K_0$ are required to fix the homotopy class of $\partial \Sigma_{g,1}$.
{We will prove in \cite{HM_tg2} that this} boundary condition implies {that the two nested filtrations \eqref{nested} on $G$ agree:}
$$
{G_m}= {G^0_m} \quad  \hbox{for all } m\geq 0.
$$
{In} this case, the Johnson homomorphisms 
$(\tau_m^1)_{m\geq 0}$  and $({\ti\tau}_m^0)_{m\geq 0} = ({\tau}_m^0)_{m\geq 0}$
are interchangeable and correspond to the ``tree reduction'' of the Kontsevich-type functor $Z$ introduced in \cite{HM_tg0}.
{Moreover,}  the {maps $\tau^1_m|_{G^0_{m+1}}$ given in}  Proposition \ref{r20} are trivial.
\end{example}

\subsection{N-series determined by $K_1$ and $K_2$}\label{sec:second-way}

Let ${K=K_1}$ be a group, and {let $K_2\trll K$ with ${K_2 \ge  [K,K]}$}.
Let $K_+ = (K_m)_{m\geq 1}$ be the smallest N-series of $K$ with these $K_1$ and $K_2$,
{i.e., $K_+$ is defined by}
\begin{equation} \label{Km}
  K_m= [K_{m-1}, K_1] \cdot  [K_{m-2}, {K_2}]
\end{equation}
inductively for $m\ge3$. Note that
$$
\Gamma _m K \subset K_m \subset \Gamma _{\lceil m/2 \rceil} K
$$
{for $m\ge 1$,}
where $\lceil m/2 \rceil=\min\{n\in\Z\;|\;n\ge m/2\}$.
Extend $K_+$ to an extended N-series $K_*=(K_m)_{m\geq 0}$ with $K_0=K_1$.

{Let} a group $G$ act on  ${K}$ in such a way that ${}^GK_2=K_2$.
{Then each $g\in G$ satisfies $g(K_j) \subset K_j$ for all $j\ge 3$,
as can be verified inductively  using \eqref{Km}. Hence $G$ acts on $K_*$
and we can consider the induced Johnson filtration $\modF^{K_*}_*(G)$. It} has the following  description.  {Set}
\begin{gather}
 \label{G[m]_bis}
{G_m} =  \{g\in G
\  \vert\ [g,K_1]\subset {K}_{m+1},\, [g, {K_2}]\subset {K}_{m+2} \} \qquad  \hbox{for $m\ge 0 $.}
\end{gather}

\begin{proposition}
  \label{r205}
  We have $ \modF ^{{K}_*}_m(G)={G_m}$ for all~${m\ge 0}$.
  Hence $G_*=(G_m)_{m\ge 0}$  is an extended N-series.
\end{proposition}

\begin{proof}
{Obviously,   $ \modF ^{K_*}_m(G) \subset G_m $ and $G_0=G=\modF^{K_*}_0(G)$.}

{It} remains to  prove ${G_m} \subset \modF ^{K_*}_m(G)$ {for $m\ge 1$.}
It suffices  to check that if  $g\in G$ satisfies $[g,K_1]\subset K_{m+1}$ and  $[g,{K_2}]\subset K_{m+2}$,
 then we have $[g,K_i]\subset K_{m+i}$  for all $i\ge 1$. This is obvious for {$i=1,2$.}
The case $i\geq 3$ is proved by an induction using~\eqref{Km},
similarly to the proof of Proposition \ref{r10} {in the case} $K_+=\Gamma _+K$.
\end{proof}

By Corollary \ref{simpler}, we have an injective morphism of eg-Lie algebras
\begin{equation} \label{1_and_2}
\bar\tau _\bu: {\bar G}_\bu \longrightarrow \Der_\bu(\bar{K}_+).
\end{equation}
In contrast with  Section~\ref{sec:first-way},
the graded Lie algebra $\bar{K}_+$ is not generated by its {degree $1$} part.
Thus, Proposition~\ref{r98} does not apply and $t_\bu:  \Der_\bu(\bar{K}_+) \to D_\bu(\bar{K}_+)$ might not be injective.
{Nonetheless, $\bar{K}_+$ is generated by its {degree $1$ and $2$} parts.
This observation motivates the following  definitions.}

  Let $L_+$ be a graded Lie algebra, and let  $A$ be a subgroup of $L_2$ such that $L_2 =[L_1,L_1]+A$.
    We define a graded  group $D_\bu(L_+,A)$ as follows.
    For $m\geq 1$,  consider the abelian group
 \begin{eqnarray*}
 D_m(L_+,A) & = & \Hom(L_1,L_{m+1}) \times  \Hom(A,L_{m+2})
 \end{eqnarray*}
   and,  for $m=0$,  set
    \begin{eqnarray*}
    D_0(L_+,A) & = &\Big\{ (u,v) \in \Aut(L_1)  \times  \Hom(A,L_2)  \\
    & &   \quad \, \Big\vert \, \begin{array}{l} \hbox{\small the map $[x_1,y_1] +a \longmapsto [u(x_1),u(y_1)]+ v(a)$} \\
    \hbox{\small defines an automorphism of $[L_1,L_1] + A=L_2$}\end{array} \Big\}.
    \end{eqnarray*}
The {subgroup}
$$
 \{ (d_1,d_2) \in \Aut(L_1) \times \Aut(L_2) \, \vert \, d_2([b,c]) = [d_1(b),d_1(c)] \hbox{ for } b,c \in L_1 \}
$$
of $\Aut(L_1) \times \Aut(L_2)$ is mapped bijectively onto
$D_0(L_+,A)$ by $(d_1,d_2) \mapsto (d_1,d_2\vert_A)$. Hence
$D_0(L_+,A)$ inherits from $\Aut(L_1) \times \Aut(L_2)$ a  group structure.
For every~${m\geq 0}$, there is a homomorphism
$$
t_m: \Der_m(L_+) \lto D_m(L_+,A), \ (d_i)_{i\geq 1} \longmapsto (d_1,d_2\vert_A).
$$
Clearly, $t_\bu=(t_m)_{m\geq 0}$ is injective
if the graded Lie algebra $L_+$ is generated by its degree $1$ and $2$ parts
(and, so, by {$L_1 \oplus A$}).
Furthermore, $t_\bu$ is bijective if $L_+$ is freely generated
by {$L_1\oplus A$, where $L_1$ and $A$ are in degree $1$ and $2$,  respectively.} Hence,
in this case, there is a unique eg-Lie algebra structure on $D_\bu(L_+,A)$ such that $t_\bu$ is an  eg-Lie algebra isomorphism.

Now, \emph{let $\bar{K}_+$ be freely generated by $B = \bar{K}_1$ and a subgroup $A$ of~$\bar K_2$.}
Then, {the previous paragraph gives}
an injective eg-Lie algebra morphism
 \begin{gather*}
   {\bar G}_\bu \stackrel{\bar\tau _\bu} {\longrightarrow} \Der_\bu(\bar K_+)  \mathop{\longrightarrow}^{t_\bu}_\simeq  D_\bu(\bar K_+,A).
  \end{gather*}
The {$m$th}  Johnson homomorphism $t_m \tau_m: G_m \to D_m(\bar K_+,A)$ has two components
$$
\tau_0^1: G_0 \longrightarrow \Aut(B), \qquad \tau_0^2: G_0 \longrightarrow {\Hom(A,\Lambda^2 B)\times \Aut(A)}
$$
for $m=0$, and
$$
\tau_m^1: G_m \longrightarrow \Hom( B, \Lie_{m+1}(B;A)), \qquad \tau_m^2: G_m \longrightarrow \Hom(A,  \Lie_{m+2}(B;A))
$$
for $m\geq 1$. Here  $\Lie(B;A)$ {denotes} the graded Lie algebra freely generated by
{$B\oplus A$, where $B$ and $A$ are in degree $1$ and $2$, respectively.}

We illustrate the above constructions with {a few}  examples.
{The following lemma} is easily deduced from \cite[Proposition 1]{Labute}.

\begin{lemma} \label{free_1_2}
{Let}  $K=K_1 =\langle x_1,\ldots ,x_p,y_1,\ldots ,y_q\rangle $ {be}
a free group of rank $p+q$ with $p,q\geq 0$, and {let}
 $$
 K_2 = \Gamma _2 K \cdot \lala  x_1,\dots, x_p \rara
 = \ker  (K \rightarrow  \langle y_1,\ldots ,y_q\rangle^\abel ).
 $$
 Then the graded Lie algebra $\bar K_+$ is freely generated by
 {$y_1K_2, \dots, y_q K_2$ in degree $1$
and by  $x_1K_3, \dots, x_p K_3$ in degree $2$.}
\end{lemma}

\begin{example}
  \label{r7}
This {generalizes}  Example \ref{Johnson}.
Let $\Sigma_{g,1}^p $ be  the surface $\Sigma_{g,1}$ with $p\geq 0$
punctures, and {let} $i:\Sigma_{g,1}^p \to \Sigma_{g,1}$ {be} the inclusion.
Set $K= K_1 = \pi_1(\Sigma_{g,1}^p,\star)$, where $\star \in \partial \Sigma_{g,1}^p = \partial \Sigma_{g,1}$, and
$$
{K_2}  = \ker\Big( {\pi_1(\Sigma_{g,1}^p,\star) \stackrel{i_*}{\longrightarrow}  \pi_1(\Sigma_{g,1},\star)
\stackrel{}{\longrightarrow} \pi_1(\Sigma_{g,1},\star)^\abel \simeq  H_1(\Sigma_{g,1};\Z)}   \Big).
$$
 The {smallest N-series $K_+=(K_m)_{m\geq 1}$ with these $K_1$ and $K_2$} is known as the
\emph{weight filtration}.  {It}  {was}  introduced by Kaneko \cite{Kaneko} in the framework of pro-$\ell$
groups following ideas of Oda, {and} has been studied by several authors
including Nakamura and Tsunogai~\cite{NT}, and Asada and Nakamura
\cite{AN}.

Set $B=K_1/K_2=H_1(\Sigma_{g,1};\Z)$ and
 $$
 A = \ker ( i_*: H_1(\Sigma_{g,1}^p;\Z) \longrightarrow H_1(\Sigma_{g,1};\Z)).
 $$
{We regard $A$} as a subgroup of $K_2/K_3$ {as follows.}
Let $x_1,\ldots ,x_p\in K$ be represented by loops (based at $\star$) around
the $p$ punctures. Since {$A$} is free abelian with basis $[x_1],\ldots ,[x_p]$, there is a unique
homomorphism $j:A \to K_2/K_3$ defined by $j([{x_i}]) = {x_i   K_3}$; one easily checks
that $j$ does not depend on the choice of $x_1,\ldots ,x_p$.
By Lemma ~\ref{free_1_2}, $j$ is injective and the graded Lie algebra $\bar K_+$ is freely
generated by ${B\oplus j(A)}$, where $B$ and $j(A)$ are in degree $1$ and $2$, respectively.

The mapping class group $G = \operatorname{MCG}(\Sigma_{g,1}^p,\partial \Sigma_{g,1}^p)$ acts on $K$ in the canonical way,
{and we have  ${}^GK_2=K_2$.}
{The} extended N-series     $G=G_0 \ge G_1 \ge G_2 \ge \cdots $
coincides with the filtration
$$
 \Gamma ^*_{g,[p+1]}  \ge  \Gamma ^*_{g,p+1}(1)  \ge  \Gamma ^*_{g,p+1}(2)  \ge  \cdots
$$
in \cite[\S 2.1]{AN}. Furthermore, for $m\geq 1$,
the Johnson homomorphism $t_m\tau_m =(\tau^1_m,\tau^2_m)$
is essentially the same  as  the homomorphism $c_m$ in \cite[\S 2.2]{AN}.

{There} is a short exact sequence
$$
1 \longrightarrow B_p(\Sigma_{g,1}) \longrightarrow G \longrightarrow
\operatorname{MCG}(\Sigma_{g,1},\partial \Sigma_{g,1}) \longrightarrow 1, 
$$
{where $B_p(\Sigma_{g,1})$ is the braid group in $\Sigma_{g,1}$ on $p$ strands.}
Thus, the  homomorphisms $\tau^i_m$ {(for $m\geq 1$, $i=1,2$)}
{generalize both} the ``classical'' Johnson homomorphisms ($p=0$)
and Milnor's $\mu$-invariants ($g=0$).
{The former  are contained in the ``tree reduction'' of the LMO functor \cite{CHM},
while the latter are contained in the ``tree reduction'' of the Kontsevich integral \cite{HM}.}
It seems possible to describe diagrammatically the  generalized Johnson homomorphisms $\tau^i_m$ for any $g,p\geq 0$
{and}  to relate them to the ``tree reduction'' of the extended LMO functor introduced in \cite{Nozaki}.
\end{example}

\begin{example}
  \label{r8}
As in Example \ref{handlebody}, consider a handlebody $V_g$ of genus $g\geq 1$ and a surface $\Sigma_{g,1} \subset \partial V_g$ of genus $g$.
Set ${K}=K_1 = \pi_1(\Sigma_{g,1},\star)$  and
$$
{K_2}  = \ker\Big( {\pi_1(\Sigma_{g,1},\star) \stackrel{i_*}{\longrightarrow}  \pi_1(V_g,\star) \stackrel{}{\longrightarrow} \pi_1(V_g,\star)^\abel \simeq  H_1(V_g;\Z)}   \Big).
$$
{The smallest N-series $K_+=(K_m)_{m\geq 1}$ with these $K_1$ and $K_2$} is given by
$$
K_2 = \Gamma _2 K \cdot \mathbf{A}, \quad K_3 = \Gamma _3 K \cdot [K, \mathbf{A}], \quad \hbox{etc.,}
$$
where $\mathbf{A}  =  \ker\big( i_*: \pi_1(\Sigma_{g,1},\star) \longrightarrow   \pi_1(V_g,\star)   \big)$.
Let
 $$
 A  =  \ker\big( i_*: H_1(\Sigma_{g,1};\Z) \longrightarrow H_1(V_{g};\Z) \big) \quad \hbox{and} \quad
 B = H_1(V_g;\Z).
 $$
 Identify $B$ with $K_1/K_2$, and  let $j:A\to K_2/K_3$  be the   canonical homomorphism
 $$
 A \simeq \frac{\Gamma _2 K \cdot \mathbf{A}}{ \Gamma _2 K} \simeq \frac{  \mathbf{A}}{\Gamma _2 K \cap \mathbf{A} }
 =  \frac{  \mathbf{A}}{[K,\mathbf{A}] }
\longrightarrow \frac{K_2}{K_3}.
$$
Then, by Lemma \ref{free_1_2}, $j$ is injective and
 the graded Lie algebra $\bar K_+$ is freely
 generated by $B\oplus j(A)$, where $B$ and $j(A)$ are in degree $1$  and $2$, respectively.

The subgroup $G$ of $\operatorname{MCG}(\Sigma_{g,1}, \partial
  \Sigma_{g,1})$ that preserves the Lagrangian subgroup $A \subset H_1(\Sigma_{g,1};\Z)$ 
is usually called the \emph{Lagrangian mapping class group} of~$\Sigma_{g,1}$.
It  acts on~$K$ in the canonical way and satisfies ${}^G K_2 = K_2$.
Hence we {obtain} 
an extended N-series $G_*=  (G_m)_{m\geq 0}$, which is the Johnson filtration induced by $K_*$.
The  generalized Johnson homomorphisms $\tau_m^i$ {(}for $m\geq 0$, $i=1,2${)}
will be studied by Vera \cite{Vera} in relation with the ``tree reduction'' of the LMO functor introduced in \cite{CHM}.
This is also connected to the ``Lagrangian'' versions of the Johnson homomorphisms introduced by Levine in \cite{Levine1,Levine2}.
\end{example}

\section{Filtrations on group rings and their associated graded}\label{sec:filtr-group-rings-1}

In this section, we consider filtrations on group rings induced by extended N-series
and we compute their associated graded.
By a \emph{ring} we mean an associative ring with unit.

\subsection{Filtrations on group rings}    \label{sec:filtr-group-rings}

A \emph{filtered ring} $J_* = (J_{m})_{m\geq 0} $ is a ring $J_0$ together with a decreasing sequence
$$
 J_0 \supset J_1 \supset \cdots \supset J_k \supset J_{k+1} \supset \cdots
$$
of additive subgroups  such that
\begin{equation}
J_m J_n \subset J_{m+n} \quad \hbox{for } m,n\geq 0.
\end{equation}
Note that $J_m$ is an ideal of $J_0$ for {each} $m\geq 0$.
The \emph{associated graded} of $J_*$,
$$
\gr_\bullet (J_*) = \bigoplus_{k\geq 0} \frac{J_k}{J_{k+1}},
$$
{has} the obvious graded ring structure.

Let $K_*$ be an extended N-series, and $\Z[K_0]$ the group ring of $K_0$.
For $m\ge 1$, we~set
\begin{gather*}
I_m({K_*}) =  \ker \big(\Z  [K_0]\overset{\Z [ \pi _m ]}{\longrightarrow}\Z [K_0/K_m]\big),
\end{gather*}
where {$\pi _m: K_0\rightarrow K_0/K_m$ is the  projection.}
We associate to $K_*$ the filtered ring
\begin{equation} \label{Delta}
J_*(K_*) = (J_m(K_*) )_{m\geq 0}
\end{equation}
defined by $J_0(K_*) = \Z[K_0]$ and  by
\begin{gather*}
  J_m(K_*)  = \sum_{\substack{m_1,\ldots ,m_p\ge 1,\ p\ge 1\\m_1+\dots +m_p\geq m}}I_{m_1}(K_*) \cdots I_{m_p}(K_*)
\quad \hbox{for $m\ge 1$.}
\end{gather*}
Note that $ J_m(K_*)$ is the ideal of $\Z[K_0]$ generated by the elements  $(x_1-1)\cdots (x_p-1)$
for all  $x_1 \in K_{m_1}, \dots, x_p \in K_{m_p}$, $m_1+\dots +m_p \geq m$, $m_1,\ldots ,m_p\ge 1,\ p\ge 1$.
For instance, if $K_*$ is the extended N-series defined by the lower central series of the group $K_0$, then we have  $J_m(K_*) =I^m$, where
$I$ is the augmentation ideal of $\Z[K_0]$.

{Now we} equip the group ring $\Z[K_0]$ with {the} usual Hopf algebra
  structure {with {comultiplication}
    $\Delta $, counit $\epsilon $ and antipode $S$.} Since
\begin{gather*}
  \Delta(I_k(K_*)) \subset
  I_k(K_*) \otimes \Z[K_0]
  + \Z[K_0]\otimes I_k(K_*)
\end{gather*}
for $k\ge 0$, we have
$$
\Delta(J_m(K_*) ) \subset \sum_{i+j =m } J_i(K_*) \otimes J_j(K_*).
$$
Clearly, we  have $\epsilon (J_m(K_*) ) =0$  and $S(J_m(K_*)) = J_m(K_*)$ for all $m\geq 1$.
Hence  $J_*(K_*)$ has the structure of a filtered Hopf algebra
and, consequently, the associated graded
$$
\gr_\bullet( J_*(K_*)) = \bigoplus_{i\geq 0} \frac{J_i(K_*)}{J_{i+1}(K_*)}
$$
has the structure of a graded Hopf algebra.

\subsection{Universal enveloping algebras of eg-Lie algebras} \label{sec:uea}

 Let $L_\bu$ be an eg-Lie algebra.  Then we have two Hopf
algebras $\modZ [L_0]$ and $U(L_+)${, the universal enveloping algebra of  $L_+$.}
The action of $L_0$ on $L_+$ induces
an action {of} $\modZ [L_0]$ on $U(L_+)$.  The \emph{universal enveloping
algebra} $U(L_\bu)$ of $L_\bu$ is defined to be the crossed product
(or the smash product) $U(L_+)\,\sharp\,\modZ [L_0]$ of $U(L_+)$ and
$\modZ [L_0]$, which is the Hopf algebra structure on $U(L_+)\otimes \modZ [L_0]$
with multiplication and comultiplication defined by
\begin{gather}
  \label{multipl}
  (u\otimes g) \cdot (u'\otimes g)=u\, ({}^gu') \otimes gg'
  \qquad \text{for $u,u'\in U(L_+)$, $g,g'\in L_0$},\\
  \label{e1}
  \Delta (u\otimes g)=  \sum     ({u'}  \otimes g)\otimes ({u''}  \otimes g)
  \qquad \text{for $u\in U(L_+)$, $g\in L_0$},
\end{gather}
where $\Delta (u)= \sum{u'\otimes u''}$.

We usually write $u\otimes g=u\cdot g$ in $U(L_\bu)$, and we regard
both $U(L_+)$ and $\modZ [L_0]$ as Hopf subalgebras of $U(L_\bu)$.
By \eqref{multipl} we have
$$
g\cdot u \cdot g^{-1} = {}^g u\qquad \text{for $g\in L_0$,  ${u\in  U(L_+)}$}.
$$
The grading of $L_+$ makes $U(L_\bu)$ a graded Hopf algebra.

\subsection{Taking rational coefficients}\label{sec:taking-rati-coeff}

{Here we carry out some of the previous constructions over $\Q$.}
First of all, there is a notion of \emph{filtered $\Q$-algebra}
similar to {that} of filtered ring in Section \ref{sec:filtr-group-rings}.
For each extended N-series $K_*$, there is a filtration $J_*^\Q(K_*)$ {of $\Q[K_0]$} whose
definition {is} parallel to {that} of $J_*(K_*)$.

We define an \emph{eg-Lie $\Q$-algebra}~$L_\bu$ in the same  way as
{an} eg-Lie algebra in Section~\ref{sec:extended-graded-lie}:
here $L_+$ is assumed to be a graded Lie algebra {over} $\Q$.
For each extended N-series $K_*$, there is an \emph{associated eg-Lie $\Q$-algebra} $\gr_\bu^\Q(K_*)$
defined by  $\gr_0^\Q(K_*) = K_0/K_1$ and $ \gr_m^\Q(K_*) = (K_m/K_{m+1}) \otimes \Q$ for $m\geq 1$.

The contents of Section \ref{sec:derivation-eg-lie} can also be adapted to an eg-Lie $\Q$-algebra $L_\bu$.
Thus  we  define  the \emph{derivation eg-Lie $\Q$-algebra} $\Der_\bu(L_\bu)$ of~$L_\bu$,
and Theorem \ref{r101} works {over $\Q$ as well.}

Finally, the definitions of Section \ref{sec:uea} {work} also {over $\Q$.}
{The} \emph{universal enveloping algebra} $U(L_\bu)$ of an
eg-Lie $\Q$-algebra~$L_\bu$ is the $\Q$-vector space
$ U(L_+)\otimes_\Q \modQ [L_0]$
with multiplication $\cdot$ defined by~\eqref{multipl}.
Note that   $U(L_\bu)$ has  a graded  Hopf $\Q$-algebra structure.  Let
$\hat U(L_\bu)$ denote its degree-completion, which is a complete Hopf algebra.

\begin{lemma} \label{glp}
For {every} eg-Lie $\Q$-algebra $L_\bu$,
the group-like part of $\hat U(L_\bu )$ is
$$
\{ \exp(\ell)\cdot g\ \vert\ \ell \in \hat L_+, g\in L_0 \},
$$
where $\hat L_+$ denotes the degree-completion of $L_+$.
\end{lemma}

\begin{proof}
  It is easy to see that  $\exp(\ell)\cdot g$ is group-like  in $\hat
  U(L_\bullet)$ for $\ell \in \hat L_+$, $g\in L_0$.

Conversely, let $x$ be a group-like element of $\hat U(L_\bu)$.
We can write
\begin{gather}
  \label{e14}
  x = \sum_{g\in L_0} x_g \cdot g,
\end{gather}
where $x_g\in \hat{U}(L_+)$ are uniquely determined by $x$, and for each
$m\geq 0$ there are only finitely many $g\in L_0$ such that the
degree $m$ part of $x_g$ is non-zero.
We have
\begin{gather*}
  \Delta (x)
  =\sum_{g\in L_0}  \sum  \big({x_g'}    \cdot g\big) \otimes  \big({x_g''}    \cdot g\big),
\end{gather*}
where  $\Delta (x_g)=\sum x_g'\otimes x_g''$.
We also have
\begin{gather*}
  x\otimes x=\sum_{g,h\in L_0}  (x_g\cdot g) \otimes  (x_h\cdot h),
\end{gather*}
Since $\Delta (x)=x\otimes x$, it follows that
\begin{eqnarray*}
  \Delta (x_g)=x_g\otimes x_g &&  \text{for {all} $g\in L_0$},\\
  x_g\otimes x_h = 0&&  \text{for {all} $g,h\in L_0$, $g\neq h$}.
\end{eqnarray*}
Since $x\neq 0$, there is $g \in L_0$ such that $x=x_g\cdot g$ and $x_g$ is group-like.
Hence $\ell = \log( x_g )$ is {primitive in} $\hat U(L_+)$.
Since the primitive part of $ U(L_+)$ is $L_+$, {the element} $\ell$ belongs to the degree-completion of $L_+$.
\end{proof}

\subsection{Quillen's description of the associated graded of a group ring}\label{sec:quill-descr-assoc}

A~well-known result of Quillen describes the associated graded of a group ring
filtered by powers of the augmentation ideal \cite{Quillen}.
{This} result is generalized to the filtration of a group ring induced by
{any} extended N-series, as follows.

\begin{theorem}   \label{Quillen}
Let $K_*$ be an extended N-series. There is a (unique) ring homomorphism
\begin{equation}  \label{Theta}
\Upsilon:   U(\gr_\bu(K_*) )  \longrightarrow \gr_\bu( J_*(K_*))
\end{equation}
 defined by $\Upsilon(gK_1)= g+ J_1(K_*) $ for $g \in K_0$ and by $\Upsilon(xK_{i+1}) = (x-1) + J_{i+1}(K_*)$ for $x \in K_i$, $i\geq 1$.
Furthermore, the rational version of $\Upsilon$
$$
\Upsilon^\Q:   U(\gr_\bu^\Q(K_*) )  \longrightarrow \gr_\bu( J_*^\Q(K_*))
$$
 is a $\Q$-algebra isomorphism.
\end{theorem}

\begin{proof}
{The N-series $K_+ = (K_m)_{m\geq 1}$ defined by $K_*$} induces a filtration
\begin{equation} \label{J'}
J'_+(K_+) = (J'_m(K_+))_{m\geq 1},
\end{equation}
where $J'_m(K_+)$ is the subgroup of $\Z[K_1]$ spanned by the elements  $(x_1-1)\cdots (x_p-1)$
for all  $x_1 \in K_{m_1}, \dots, x_p \in K_{m_p}$, $m_1+\dots +m_p \geq  m$, $m_1,\ldots ,m_p\ge 1,\ p\ge 1$.
({It} is an ideal of $\Z[K_1]$ contained in $J_m(K_*)$.) Let
$$
\gr_+ (J'_+(K_+))  = \bigoplus_{m\geq 1} \frac{J'_m(K_+)}{J'_{m+1}(K_+)}
$$
be the {associated} graded ring, and let
$$
\gr_+(K_+) =  \bigoplus_{m\geq 1} \frac{K_m}{K_{m+1}}
$$
be the graded Lie algebra associated to the  $N$-series $K_+$.
It is easily checked that the graded abelian group homomorphism
$$
\gr_+(K_+) \longrightarrow \gr_+( J'_+(K_+)), \ (xK_{m+1}) \longmapsto (x-1)+J'_{m+1}(K_+)
$$
preserves the Lie bracket {and hence}
induces a  ring homomorphism
$$
\Upsilon':   U(\gr_+(K_+) )  \longrightarrow \gr_+( J'_+(K_+)).
$$
By composing it with the canonical map  $\gr_+( J'_+(K_+)) \to  \gr_+( J_*(K_*))$,
we obtain a ring homomorphism
\begin{equation} \label{Theta1}
\Upsilon:  U(\gr_+ (K_*) ) = U(\gr_+(K_+) )  \longrightarrow  \gr_\bu( J_*(K_*)).
\end{equation}
Besides, the inverse of the canonical isomorphism $\Z[K_0]/J_1(K_*) \to \Z[\bar K_0]$, where $\bar K_0  =  K_0/K_1$, defines a ring homomorphism
\begin{equation} \label{Theta2}
\Upsilon:\Z[\bar K_0] \longrightarrow \gr_\bu( J_*(K_*)).
\end{equation}
A straightforward computation shows that \eqref{Theta1} and
\eqref{Theta2} define together a  ring homo\-morphism~\eqref{Theta} on
{$U(\gr_\bu(K_*))=U(\gr_+ (K_*) )\,\sharp\,\Z[\bar K_0]$}.

As a generalization of {Quillen's result mentioned above,}
it is known that the rational version $\Upsilon'^{\Q}$ of $\Upsilon'$ is an isomorphism \cite[Corollary 5.4]{Massuyeau_DSP}.
Thus, to conclude that $\Upsilon^\Q$ is  an isomorphism,
it {suffices} to prove that  $\gr_+ ( J_*(K_*))$ is isomorphic to $\gr_+( J'_+(K_+))
\otimes \Z[\bar K_0]$. Specifically, we need to prove that the group homomorphism
$$
r: \frac{J'_m(K_+)}{ J'_{m+1}(K_+)} \otimes  \Z[\bar K_0] \longrightarrow \frac{J_m(K_*)}{ J_{m+1}(K_*)}
$$
defined by $r((u + J'_{m+1}(K_+)) \otimes (gK_1) )= (ug + J_{m+1}(K_*))$
is an isomorphism {for each $m\geq 1$}. Clearly, $r$ is surjective. To construct a left inverse to $r$,
{let $\pi: K_0 \to \bar K_0$ denote} the canonical projection,
and let $s:\bar K_0 \to K_0$ be a set-theoretic section of~$\pi$. Then there is a unique group  homomorphism
$$
q: \Z[K_0] \longrightarrow \Z[K_1] \otimes \Z[\bar K_0]
$$
defined by $q(g) = (g\, (s\pi(g))^{-1}) \otimes \pi(g)$ for $g\in K_0$. For any
$x_1 \in K_{m_1}, \dots, x_p \in K_{m_p}$ with $m_1+\dots +m_p \geq  m$, $m_1,\ldots ,m_p\ge 1,\ p\ge 1$, and for any $y\in K_0$, we have
  $$
  q\Big( (x_1-1)\cdots (x_p-1) y \Big) = (x_1-1)\cdots (x_p-1) \big( y (s\pi(y))^{-1}\big) \otimes \pi(y),
  $$
which shows that $q(J_m(K_*)) \subset J'_m(K_+) \otimes \Z[\bar K_0]$. Therefore, $q$ induces a  group homomorphism
$$
q: \frac{J_m(K_*)}{ J_{m+1}(K_*)} \longrightarrow \frac{J'_m(K_+) \otimes \Z[\bar K_0]}{ J'_{m+1}(K_+) \otimes \Z[\bar K_0]}
\simeq  \frac{J'_m(K_+) }{ J'_{m+1}(K_+) } \otimes \Z[\bar K_0],
$$
which satisfies $qr=\id$.
\end{proof}

\begin{remark}
{It is easily verified} that $\Upsilon$ preserves the {graded} Hopf algebra structures.
Hence $\Upsilon_\Q$ is a {graded} Hopf $\Q$-algebra isomorphism.
\end{remark}

\section{Formality of extended N-series}\label{sec:formality-extended-n}

Assuming that an extended N-series $K_*$ is ``formal'' in some sense,
we {here} show  that an action of an extended N-series $G_*$ on $K_*$ has an
``infinitesimal'' counterpart {containing}
all the Johnson homomorphisms. In this section, we {work over $\Q$.}

\subsection{Formality and expansions}\label{sec:formality-expansions}

Let $K_*$ be an extended N-series and consider the completion
$$
\wh{\Q [K_*]} = \plim_{k}\  \Q[K_0]\, / J_k^\Q(K_*)
$$
of the group $\Q$-algebra $\Q [K_0]$ with respect to the rational version $J_*^\Q(K_*)$ of the filtration~\eqref{Delta}.
The filtered Hopf $\Q$-algebra structure of $\Q [K_0]$  extends to a complete Hopf algebra structure on $\wh{\Q [K_*]}$,
whose filtration is denoted by $\hat{J}_*^\Q(K_*)$.

An extended N-series $K_*$ is said to be \emph{formal} if the complete Hopf algebra
$\wh{\Q [K_*]}$ is isomorphic to the degree-completion  of its associated graded, namely
$$
\widehat{\gr}_\bu( J_*^\Q(K_*)) = \prod_{k\geq 0} \frac{J_k^\Q(K_*)}{J_{k+1}^\Q(K_*)},
$$
through an isomorphism {whose associated graded is the identity}.

Recall that $\hat U(  \gr_\bu^\Q(K_*) )$ denotes the degree-completion of the universal enveloping algebra
 of the eg-Lie $\Q$-algebra $\gr_\bu^\Q(K_*)$ associated to the extended N-series~$K_*$.
An \emph{expansion} of an extended N-series $K_*$ is a {homomorphism}
$$
\theta: K_0 \longrightarrow \hat U(  \gr_\bu^\Q(K_*) )
$$
which maps any $x\in K_i$, $i \geq 0$  to a group-like element of the form
\begin{equation} \label{theta}
\theta(x) = \left\{\begin{array}{ll}
1+ (xK_{i+1}) + (\deg >i) & \hbox{if } i>0,\\
(xK_1) + (\deg>0) & \hbox{if } i=0.
\end{array}\right.
\end{equation}

\begin{example} \label{expansion_free}
Assume that $K_*$ is associated with the lower central series of a free group $ K_0=K_1$.
Let $\Lie(H^\Q)$ denote the free Lie $\Q$-algebra generated by $H^\Q = (K_1/K_2)\otimes \Q$ in degree $1$.
Then the identity of $H^\Q$ extends uniquely to an isomorphism $\Lie(H^\Q) \simeq \gr_+^\Q(K_*)$ of graded Lie $\Q$-algebras,
so that we have a canonical isomorphism of graded Hopf $\Q$-algebras
$$
U(  \gr_\bu^\Q(K_*)) =   U(\gr_+^\Q(K_*)) \simeq U( \Lie(H^\Q)) = T(H^\Q),
$$
where $T(H^\Q)$ is the tensor algebra generated by $H^\Q$ in degree $1$.
Hence, in this case, an expansion of $K_*$ is a {homomorphism} $\theta: K_0 \to \hat T(H^\Q)$ such that
\begin{equation}  \label{theta(x)}
\theta(x) = \exp\Big( [x] + (\hbox{series of Lie elements of degree $>1$)}\Big)
\end{equation}
for all $x\in K_0$, where $[x]  =  (x K_2) \otimes 1 \in H^\Q$.
For instance,  for each basis $b=(b_i)_{i}$ of $K_0$,
there is a unique expansion $\theta_b$ of $K_*$
{such that}  $\theta_b(b_i) = \exp([b_i])$.
\end{example}

The following establishes the relationship between formality and expansions.

\begin{proposition} \label{formal_expansion}
An extended N-series $K_*$ is formal if and only if it has an expansion.
\end{proposition}

\begin{proof}
{Consider} the diagram
\begin{gather}
  \label{diag1}
  \xymatrix{
    K_0
    \ar@{.>}[r]^{\theta\quad }
    \ar[d]_{\iota}
    &
    \hat U(  \gr_\bu^\Q(K_*) )
    \ar[d]^{\hat \Upsilon^\Q}_{\simeq}
    \\
    \wh{\Q [K_*]}
    \ar@{.>}[r]^{f\quad }_{\simeq\quad}
    \ar@{.>}[ru]^{\hat\theta}_{\simeq}
    &
    \widehat{\gr}_\bu( J_*^\Q(K_*)),
    }
\end{gather}
where $\iota$ is the canonical map and
${\hat \Upsilon^\Q}$ is the isomorphism {in} Theorem \ref{Quillen}.

Assume that $K_*$ is formal. Then there is a complete Hopf algebra
isomorphism $f$ in \eqref{diag1} inducing the identity on the
associated graded.  The complete Hopf algebra isomorphism
$\hat\theta:=({\hat \Upsilon^\Q})^{-1}f$ satisfies
 \begin{equation*}
 \hat\theta(y) = (\Upsilon^\Q)^{-1}\Big(y+ J_{m+1}^\Q(K_*)\Big) + (\deg >m)  \quad \hbox{for $y\in J_m^\Q(K_*)$, $m\geq 0$},
 \end{equation*}
 which implies \eqref{theta} for the {homomorphism} $\theta:=\hat\theta\iota$.
 Since $\iota(K_0)$ is contained in the group-like part of $\wh{\Q [K_*]}$ and $\hat \theta$ preserves the {comultiplication},
 $\theta(K_0)$ is contained in the group-like part of $\hat U(\gr_\bu^\Q (K_*))$.

 Conversely, assume that $K_*$ has an expansion, i.e., a
 {homomorphism} $\theta$ in \eqref{diag1}.
Extend $\theta$ by linearity to an algebra homomorphism $\theta: \Q[K_0] \to \hat U(  \gr_\bu^\Q(K_*) ) $,
which is filtration-preserving by \eqref{theta}.  Hence it
induces a complete algebra homomorphism
$\hat\theta$ in \eqref{diag1}.
Since $\iota(K_0)$ generates $\wh{\Q [K_*]}$ as a topological vector space
and since  $\hat \theta$ maps $\iota(K_0)$ into the group-like part of $\hat U(  \gr_\bu^\Q(K_*) )$,
it follows that $\hat \theta$ preserves the {comultiplication}: therefore, $\hat \theta$ is a complete Hopf algebra  homomorphism.
By \eqref{theta}, $\hat \theta$ induces the isomorphism
 $(\Upsilon^\Q)^{-1}$ {on the associated graded}:
 hence $\hat \theta$ is an isomorphism. Thus, $f:=\Upsilon^\Q \hat
 \theta$ tells us that $K_*$ is formal.
\end{proof}

\begin{remark} \label{smaller_theta}
Let $\theta$ be an expansion of an extended N-series $K_*$.
{The arguments in the proof of Proposition \ref{formal_expansion} shows}
that $\theta$ induces a complete Hopf algebra isomorphism
$$
\hat\theta: \wh{\Q [K_+]} \longrightarrow \hat U(  \gr_+^\Q(K_*) ),
$$
where $ \wh{\Q [K_+]}$ denotes the completion of $\Q[K_1]$
with respect to the rational version of the filtration $J'_+(K_+)$ defined at \eqref{J'}.
\end{remark}

\begin{remark}
Assume that $K_*$ is the extended N-series defined by the lower central series of a group.
Then an expansion of $K_*$ in our sense is called a ``Taylor expansion'' in \cite{BN}
and a ``group-like expansion'' in \cite{Massuyeau_IMH} (in the case of a free group).
Note that $K_*$ is formal in our sense if and only if it is ``filtered-formal'' (over $\Q$) in the sense of \cite{SW}.
Proposition \ref{formal_expansion} is a generalization of
\cite[Proposition 2.10]{Massuyeau_IMH} and \cite[Theorem 8.5]{SW}.
\end{remark}

\subsection{Actions of extended N-series in the formal case}\label{sec:actions-extended-n}

Let a group $G$ act on an extended N-series $K_*$. This action induces a homomorphism
$$
\rho: G \longrightarrow \Aut( \wh{\Q [K_*]} )
$$
with values in the automorphism group of the complete Hopf algebra  $\wh{\Q [K_*]}$. Here,
$\rho$ maps {each} $g\in G$ to  the unique automorphism  $\rho(g)$
extending the automorphism of $K_0$ defined by $x\mapsto {}^gx$.

Now, assume that $K_*$ is formal, and fix an expansion
$\theta$ of $K_*$. According to the proof of Proposition
\ref{formal_expansion},
$\theta$ extends uniquely to a complete Hopf algebra isomorphism
$$
\hat \theta:  \wh{\Q [K_*]} \longrightarrow  \hat U(  \bar K_\bu^\Q ),
$$
where $U(  \bar K_\bu^\Q )$ is the universal enveloping algebra
of the eg-Lie $\Q$-algebra $\bar K_\bu^\Q:=\gr_\bu^\Q(K_*)$ associated to the extended N-series $K_*$.
Thus $\theta$ induces a  homomorphism
$$
\rho^\theta : G \longrightarrow \Aut(  \hat U(  \bar K_\bu^\Q ) )
$$
defined by $\rho^\theta(g) = \hat \theta \rho(g) \hat\theta^{-1}$ for $g \in G$.

{Furthermore, we assume that $G$ is equipped with an N-series $G_+=(G_m)_{m\geq 1}$
and that (the extended N-series corresponding to) $G_+$ acts on $K_*$.}
Recall that $\Der_+(\bar K_\bu^\Q)$ denotes {the derivation graded Lie algebra of}
the eg-Lie $\Q$-algebra $\bar K_\bu^\Q$, and let $\widehat{ \Der}_+(\bar K_\bu^\Q)$ denote its degree-completion.
Here is the main construction of this section:

\begin{lemma} \label{der_+}
Let {an N-series $G_+$ of a group $G$} act on a formal extended N-series~$K_*$, and let $\theta$ be an expansion of $K_*$.
Then, for any $g\in G_m, m\geq 1$, the series
$$
\log (\rho^\theta(g))  =  \sum_{k\geq 1} \frac{(-1)^{k+1}}{k}(\rho^\theta(g)-\id)^{k} \in \End_\Q (\hat U( \bar K_\bu^\Q )  )
$$
converges and its restriction  to $\bar K^\Q_0=K_0/K_1$ and $\bar K_+^\Q = \bar K_+ \otimes \Q$
defines an element {$\varrho^\theta(g)$} of the degree $\geq m$ part of $\widehat{\Der}_+(\bar K_\bu^\Q)$.
\end{lemma}

\begin{proof}
Let $g\in G_m,m\geq 1$ and let $r = \rho(g)$.  Since
$$
r(x) =  x+(r(x)x^{-1} -1)x \ \in \ \big( x+ J_{m}^\Q(K_*) \big)
\quad  \text{for $x\in K_0$},
$$
we have
$(r-\id)(\wh{\Q [K_*]}) \subset \hat J_{m}^\Q(K_*)$; similarly, since
$$
r(x -1)= (x-1)+{(r(x)x^{-1}-1 )x} \ \in \  \big(  (x-1)+J_{i+m}^\Q(K_*) \big)
\quad \text{for  $x\in K_i$, $i\geq 1$,}
$$
we have $(r-\id)(\hat J_n^\Q(K_*))  \subset \hat J_{n+m}^\Q(K_*)$ for all $n \geq 1$.
Hence
\begin{equation} \label{increase}
(r-\id)^p(\hat J_n^\Q(K_*))  \subset \hat J_{n+pm}^\Q(K_*)
  \quad \hbox{for all $n\geq 0$, $p\geq 1$}.
\end{equation}
Taking $n=0$ in \eqref{increase}, we {see}
that
$$
\log(r) = \sum_{k\geq 1} \frac{(-1)^{k+1}}{k}(r-\id)^{k}
$$
is well defined as a linear endomorphism of $\wh{\Q [K_*]}$
and, taking $p=1$  in \eqref{increase}, we {see}
that $\log(r)$ increases the filtration step by $m$:
$$
\log(r)(\hat J_n^\Q(K_*)) \subset \hat J_{n+m}^\Q(K_*) \quad \hbox{for all }  n\geq 0.
$$
Furthermore, since $r$ is an algebra automorphism, $\log(r)$ is a
derivation of the algebra $\wh{\Q [K_*]}$.
(It is {well known}
that the logarithm of an algebra automorphism is a derivation whenever
it is defined; see e.g.\ \cite[Theorem 4]{Praagman}, whose combinatorial argument given for a commutative algebra works in general.)

Of course, the conclusions of the previous paragraph for $r$ apply to $r^\theta:=\rho^\theta(g)$ as well.
Thus we obtain
\begin{equation} \label{increase_bis}
(r^\theta-\id)^p( \hat U_{\geq n}(  \bar K_\bu^\Q )  )
\subset   \hat U_{\geq n+pm}(  \bar K_\bu^\Q )      \quad
\hbox{for all $n\geq 0$, $p\geq 1$,}
\end{equation}
and $\log(r^\theta)$ is a well-defined  derivation of the algebra  $\hat U( \bar K_\bu^\Q ) $
which increases the filtration step by $m$.

{Now we} prove that  $\log(r^\theta)$ maps  $\hat U(\bar K_+^\Q)\cdot x$ into itself for {each} ${x\in \bar K_0}$:
it suffices to prove the same property for $r^\theta$.
As a topological vector space,   $\hat U(\bar K_+^\Q)$ is spanned by its group-like elements:
for instance, this follows from Remark~\ref{smaller_theta} since $\widehat{\Q[K_+]}$ is spanned by the homomorphic image of $K_1$ as a topological vector space.
Therefore, it suffices to check  $r^\theta(u\cdot x) \in \hat U(\bar K_+^\Q)\cdot x$
for any group-like  $u \in  \hat U(\bar K_+^\Q)$. Since $u\cdot x$ is group-like, $r^\theta(u\cdot x)$ is group-like and, by Lemma \ref{glp},
we  have
$$
r^\theta(u\cdot x) = \exp(\ell)\cdot y = y + \ell\cdot  y + \frac{1}{2} \ell^2\cdot  y + \cdots
$$
for some $\ell$ in the degree-completion $\hat {\bar K}_+^{\Q}$ of $\bar K_+^{\Q}$ and $y\in \bar K_0$.
{Property \eqref{increase_bis}} with $p =1$ shows that $r^\theta$
 induces the identity {on the associated graded}.
Hence $r^\theta(u\cdot x)$  and $u\cdot x$ have the same degree $0$ part,  and we deduce that $y=x$.

Next, we show that  $\log(r^\theta)$ maps any $x\in \bar K_0$ {into}  $\hat {\bar K}_+^\Q \cdot x$.
By the previous paragraph, we have $\log(r^\theta) (x) = t x$ for some $t \in \hat U(\bar K_+^\Q)$.
Thus we need to show that $t$ is primitive.
Since $r^\theta$ is a coalgebra homomorphism, $\log(r^\theta)$ is a coderivation. It follows that
\begin{eqnarray*}
\Delta (t x) &=& \Big(\log(r^\theta) \otimes \id + \id \otimes \log(r^\theta)\Big) \Delta(x) \\
&=& \Big(\log(r^\theta) \otimes \id + \id \otimes \log(r^\theta)\Big) ( x \ho x )
\ = \   tx \ho x + x \ho tx
\end{eqnarray*}
and we deduce that $\Delta(t) = t \ho 1 +  1 \ho t$.
Similarly, we can show that $\log(r^\theta)$ maps any  $\ell \in {\bar K}_+^\Q$ to  $\hat {\bar K}_+^\Q$:
indeed, by the previous paragraph, we know that $\log(r^\theta)(\ell)$ belongs to $\hat U(\bar K_+^\Q)$
and, using that  $\log(r^\theta)$ is a coderivation, it is easily checked that $\log(r^\theta)(\ell)$ is primitive.

Thus, by the previous paragraph, we {can}
define a map $d_0: \bar K_0  \to \hat {\bar K}^\Q_+ $
and a group homomorphism $d_+: \bar K_+^\Q \to \hat {\bar K}_+^\Q$   by
$$
 \log(r^\theta)(x) = d_0(x) \cdot x  \quad \hbox{ and } \quad \log(r^\theta)(\ell) = d_+(\ell),
$$
respectively.
It remains to show that $(d_0,d_+)$ is an element of $\widehat{\Der}_+(\bar K_\bu^\Q)$,
i.e., $(d_0,d_+)$ is an infinite sum of derivations of the eg-Lie $\Q$-algebra $\bar K_\bu^\Q$.
(Those derivations will have degree $\geq m$ since  we have seen that $\log(r^\theta)$ increases the filtration step by $m$.)

First, $d_+$ consists of  derivations (in the usual sense) of the Lie $\Q$-algebra $\bar K_+^\Q$
since it is a restriction of the derivation  $\log(r^\theta)$ of the
algebra $\hat U(  \bar K_\bu^\Q )$.
Next, we check that $d_0$ is a $1$-cocycle. For any $x,y\in \bar K_0$, we have
\begin{eqnarray*}
\log(r^\theta)(xy) &=& x \cdot \log(r^\theta)(y) + \log(r^\theta)(x) \cdot y \\
&=&  x\cdot d_0(y) \cdot y + d_0(x) \cdot x \cdot y \ = \ \big({}^x d_0(y) + d_0(x) \big) \cdot xy,
\end{eqnarray*}
which shows that $d_0(xy) =  d_0(x)+ {}^x d_0(y)$. Finally,
for any $x\in \bar K_0$ and $\ell \in  {\bar K}_+^\Q$, we~have
\begin{eqnarray*}
\log(r^\theta)({}^x \ell) &=&\log(r^\theta)\big( x \cdot \ell \cdot x^{-1}\big) \\
&=& \log(r^\theta)(x) \cdot \ell \cdot x^{-1} + x \cdot \log(r^\theta)(\ell ) \cdot x^{-1} + x\cdot \ell \cdot \log(r^\theta)(x^{-1}) \\
&=& d_0(x) \cdot {}^x \ell +  {}^x d_+(\ell) - x \cdot \ell \cdot x^{-1} \cdot \log(r^\theta)(x) \cdot x^{-1}\\
&=& d_0(x) \cdot {}^x \ell +  {}^x d_+(\ell) - {}^x \ell \cdot d_0(x),
\end{eqnarray*}
which shows that $d_+({}^x \ell) = [d_0(x) , {}^x \ell ] +  {}^x d_+(\ell)$.
We conclude that $\varrho^\theta(g):=(d_0,d_+)$ belongs to $\widehat{\Der}_+(\bar K_\bu^\Q)$.
\end{proof}

We can now prove the main result of this section.

\begin{theorem} \label{infinitesimal_action}
Let an N-series $G_+$ of a group $G$ act
on a formal extended N-series~$K_*$ {with an expansion $\theta$.}
Then the filtration-preserving map
$$
\varrho^\theta: G \longrightarrow \widehat{\Der}_+(\bar K_\bu^\Q)
$$
{in} Lemma \ref{der_+} induces  the rational version of the Johnson morphism:
$$
\gr (\varrho^\theta) = \bar \tau_+^\Q : \bar G_+  \longrightarrow {\Der}_+(\bar K_\bu^\Q).
$$
\end{theorem}

\begin{proof}
Let $g\in G_m$, $m\geq 1$. {Set} $r= \rho(g)$ and $r^\theta = \rho^\theta(g)$.
The leading term of $ \varrho^\theta(g)$ is a derivation of degree $m$ of the eg-Lie $\Q$-algebra $\bar K_\bu^\Q$,
which {is denoted} by $d=(d_i)_{i\geq 0}$.

We prove that $d_0: \bar K_0 \to \bar K_m\otimes \Q$ is the rationalization of $\tau_m(g)_0: \bar K_0 \to \bar K_m$.
Let $x\in  \bar K_0$.  By definition of $d_0$, we have
$$
 \log(r^\theta)(x) \cdot x^{-1}= d_0(x) + (\deg >m) \ \in \hat{\bar K}_+^\Q.
$$
Besides, it follows from \eqref{increase_bis} that
$$
 \log(r^\theta)(x) = (r^\theta(x) -x) + (\deg >m) \ \in \hat U\big(  \bar K_\bu^\Q \big);
$$
hence
\begin{eqnarray*}
d_0(x)  &=& \big(\, \hbox{degree $m$ part of } (r^\theta(x)\cdot x^{-1} -1)\big).
\end{eqnarray*}
Let $y\in K_0$ be a representative of $x$: since $\theta(y) = x + (\deg \geq 1)$ by \eqref{theta},
we have $\hat\theta^{-1}(x) = \iota(y) z $,
where {$z \in (1 + \hat J_1^\Q(K_*))$.}
Therefore,
\begin{eqnarray*}
\hat\theta^{-1}\big(r^\theta(x)\cdot x^{-1} -1\big) & = & r\big(\hat\theta^{-1}(x)\big) \, \big(\hat\theta^{-1}(x)\big)^{-1} -1 \\
&=& r(\iota(y)) r(z) z^{-1} \iota(y)^{-1} -1.
\end{eqnarray*}
However, \eqref{increase} shows that $r(z) - z \in \hat J_{m+1}^\Q(K_*)$,
which implies that $r(z) z^{-1}$ is congruent to $1$ modulo  $\hat J_{m+1}^\Q(K_*)$.
It follows that
\begin{eqnarray*}
\hat\theta^{-1}\big(r^\theta(x)\cdot x^{-1} -1\big)  &\equiv& r(\iota(y)) \iota(y)^{-1} -1  \pmod{\hat J_{m+1}^\Q(K_*)}.
\end{eqnarray*}
We deduce that
$$
d_0(x) =\big(\, \hbox{degree $m$ part of } (\theta([g,y])-1)\, \big)
 \stackrel{\eqref{theta}}{=}
([g,y]K_{m+1}) = \tau_m(g)_0(x).
$$

Let $i\geq 1$. Now we prove that $d_i: \bar K_i \otimes \Q \to \bar K_{i+m}\otimes \Q$ is the rationalization of $\tau_m(g)_i: \bar K_i \to \bar K_{i+m}$.
Let $\ell \in  \bar K_i$. By definition of $d_i$, we have
$$
\log(r^\theta)(\ell) = d_i (\ell) + (\deg>i+m)
$$
Besides, it follows from \eqref{increase_bis} that
$$
 \log(r^\theta)(\ell) = (r^\theta(\ell) -\ell) + (\deg >i+m) \ \in \hat U(\bar K_\bu^\Q);
$$
hence
\begin{eqnarray*}
d_i(\ell)  &=& \big(\, \hbox{degree $(i+m)$ part of } (r^\theta(\ell) -\ell)\big).
\end{eqnarray*}
Let $y\in K_i$ be a representative of $\ell$. Then we have $\theta(y) = 1+ \ell +(\deg >i)$ by \eqref{theta},
which implies that $\hat \theta^{-1}(\ell) \equiv  \big(\iota(y)-1\big)  \pmod{\hat J_{i+1}^\Q(K_*)}$.
Using \eqref{increase}, we deduce that
\begin{eqnarray*}
\hat\theta^{-1}\big(r^\theta(\ell) -\ell\big) \ = \ (r-\id) \big(\hat \theta^{-1}(\ell)\big) &\equiv &(r-\id) \big(\iota(y)-1\big) \pmod{\hat J_{m+ i+1}^\Q(K_*)} \\
&=& r(\iota(y)) -\iota(y) \\
&\equiv &   r(\iota(y))(\iota(y))^{-1} -1 \pmod{\hat J_{m+ i+1}^\Q(K_*)}.
\end{eqnarray*}
We conclude that
$$
d_i(\ell)  = \big(\, \hbox{degree $(i+m)$ part of } \theta([g,y]-1)\, \big)  \stackrel{\eqref{theta}}{=}
([g,y]K_{i+m+1}) = \tau_m(g)_i(\ell).
$$
\end{proof}

\begin{remark}\label{infinitesimalversion}
{We can regard the} map $\varrho^\theta: G \to \widehat{\Der}_+(\bar K_\bu^\Q)$ {in} Theorem \ref{infinitesimal_action}
as a ``linearization'' or an ``infinitesimal version'' of the extended N-series action of~$G_+$ on~$K_*$.
Let $\widehat{\Der}_+(\bar K_\bu^\Q)_{\operatorname{BCH}}$
denote  the  group whose underlying set is $ \widehat{\Der}_+(\bar K_\bu^\Q)$
and whose multiplication $\cdot$ is defined by the Baker--Campbell--Hausdorff series:
$$
d \cdot e:=d + e+  \frac{1}{2}[d,e] + \frac{1}{12} [d,[d,e]] + \frac{1}{12} [e,[e,d]] + \cdots
\quad \hbox{for $d,e \in \widehat{\Der}_+(\bar K_\bu^\Q)$}.
$$
{(Here $[\cdot,\cdot]$ denotes the degree-completion of the Lie bracket defined in {Theorem}~\ref{r2}.)}
{Then}
 $$
 \varrho^\theta: G  \longrightarrow  \widehat{\Der}_+(\bar K_\bu^\Q)_{\operatorname{BCH}}
 $$
is a group homomorphism,
which maps $G_+$ into the N-series of $\widehat{\Der}_+(\bar K_\bu^\Q)_{\operatorname{BCH}}$
whose {$m$th} term is $\widehat{\Der}_{\geq m}(\bar K_\bu^\Q)$ for every $m\geq 1$.
\end{remark}

\begin{remark}  \label{N_0_rem}
In Theorem \ref{infinitesimal_action}, {let} $K_+$ {be} an N$_0$-series of $K_1$
(see Section~\ref{sec:0-restricted}).
 Then the canonical map $\bar K_+ \to \bar K_+^\Q$ is
 injective.  Therefore, one {can trade}
the Johnson morphism $\bar \tau_\bu$ with its rational version $\bar \tau_\bu^\Q$ {without loss of information}.
It follows that the map $\varrho^\theta$ in Theorem \ref{infinitesimal_action} determines all the Johnson homomorphisms.
\end{remark}

\begin{example}\label{r232}
Assume as in Example \ref{Johnson} that $K_*$ is the extended N-series associated with the lower central series of $K_0=K_1:=\pi_1(\Sigma_{g,1},\star)$,
and let   $G_*$ denote the {``classical''} Johnson filtration of $G_0:=\operatorname{MCG}(\Sigma_{g,1},\partial \Sigma_{g,1})$.
Then, by Proposition~\ref{N_0}, $G_+$~is an N$_0$-series of $G:=G_1$, namely  the {Torelli group} of $\Sigma_{g,1}$.
Since $K_0$ is a free group, Example~\ref{expansion_free} applies: an
expansion of $K_*$ is a {homomorphism}
$$
\theta: K_0 \longrightarrow \hat T(H^\Q), \quad \hbox{where } H^\Q  =  H_1(\Sigma;\Q)
$$
satisfying \eqref{theta(x)}. According to Remark \ref{N_0_rem},
the map $\varrho^\theta$ {in} Theorem~\ref{infinitesimal_action} contains all the ``classical'' Johnson homomorphisms.
It is shown in \cite{Massuyeau_IMH} that, for an appropriate expansion $\theta$, the map $\varrho^\theta$ can be identified
with the ``tree reduction'' of the LMO functor introduced in \cite{CHM}.
\end{example}

\end{document}